\documentclass{article}
\usepackage[utf8]{inputenc}
\usepackage[margin=1in]{geometry}
\usepackage{amsmath} 
\usepackage{amsthm}
\usepackage{todonotes}
\usepackage[nolist]{acronym}
\usepackage{xcolor}
\usepackage{amsmath}
\usepackage{amssymb}
\usepackage{mathtools}
\usepackage{subcaption}
\usepackage{hyperref}
\hypersetup{colorlinks=true, linkcolor=black}
\usepackage[ruled]{algorithm2e}
\usepackage{mathrsfs}

\newcommand{\R}{\mathbb{R}} 
\newcommand{\K}{\mathbb{K}} 
 
\newcommand{\N}{\mathbb{N}}

\newcommand{\Lie}{\mathcal{L}}

\newcommand{\M}{\mathbb{M}}

\newcommand{\A}{\mathcal{A}}
\newcommand{\0}{\mathbf{0}}

\newcommand{\bm}{\mathbf{m}}

\newcommand{\cs}{\mathcal{C}}

\DeclarePairedDelimiter{\abs}{\lvert}{\rvert}
\DeclarePairedDelimiter{\norm}{\lVert}{\rVert}

\DeclarePairedDelimiter{\Tr}{\textrm{Tr}(}{)}

\DeclarePairedDelimiterX{\inp}[2]{\langle}{\rangle}{#1, #2}
\DeclarePairedDelimiter{\supp}{\textrm{supp}(}{)}
\DeclarePairedDelimiter{\Mp}{\mathcal{M}_+(}{)}

\newtheorem{thm}{Theorem}[section]
\newtheorem{lem}[thm]{Lemma}

\newtheorem{prob}[thm]{Problem}
\newtheorem{cor}{Corollary}
\newtheorem{defn}{Definition}[section]

\newtheorem{rmk}{Remark} 
\newcommand{\levy}{L\'{e}vy }
\newcommand{\ito}{It\^{o} }

\newcommand{\cX}{\mathcal{X}}
\newcommand{\cY}{\mathcal{Y}}
\newcommand{\Sym}{\mathbb{S}}

\newcommand{\bbmu}{\boldsymbol{\mu}}
\newcommand{\bell}{\boldsymbol{\ell}}

\newlength{\exfiglength}
\title{\LARGE \bf Peak Value-at-Risk Estimation of Stochastic Processes using Occupation Measures
}
\renewcommand\footnotemark{}

\author{Jared Miller$^{1,3}$,  Matteo Tacchi $^2$, Mario Sznaier$^1$, Ashkan Jasour$^3$
\thanks{$^1$ J. Miller is with the Automatic Control Laboratory (IfA),  ETH Z\"{u}rich, Physikstrasse 3, 8092, Z\"{u}rich, Switzerland (e-mail: jarmiller@control.ee.ethz.ch).}
\thanks{$^2$M. Tacchi is with Automatic Control Laboratory, EPFL, 1015 Lausanne, Switzerland, and Univ. Grenoble Alpes, CNRS, Grenoble INP (Institute of Engineering Univ. Grenoble Alpes), GIPSA-lab, 38000 Grenoble, France. (e-mail: matteo.tacchi@gipsa-lab.fr).}
\thanks{$^3$J. Miller, and M. Sznaier are with the Robust Systems Lab,  ECE Department, Northeastern University, Boston, MA 02115. (msznaier@coe.neu.edu).}
\thanks{$^3$ A. Jasour is with the Team 347T-Robotic Aerial Mobility, Jet Propulsion Lab, Pasadena, CA, 91109. (e-mail: jasour@jpl.caltech.edu).}
\thanks{J. Miller and M. Sznaier were partially supported by NSF grants  CNS--1646121, ECCS--1808381 and CNS--2038493, AFOSR grant FA9550-19-1-0005, and ONR grant N00014-21-1-2431.  
J. Miller was in part supported by the Chateaubriand Fellowship of the Office for Science \& Technology of the Embassy of France in the United States and by Swiss National Science Foundation Grant 200021\_178890.
M. Tacchi was supported by the French company R\'{e}seau de Transport d'\'{E}lectricit\'{e}, as well as the Swiss National Science Foundation under the “NCCR Automation” grant n$^{\circ}$51NF40\_{}180545.
}
}

\begin{document}

\maketitle


\begin{abstract}
\label{sec:abstract}
This paper formulates algorithms to upper-bound the maximum Value-at-Risk (VaR) of a state function along trajectories of stochastic processes. The VaR is upper bounded by two methods: minimax tail-bounds (Cantelli/Vysochanskij-Petunin) and Expected Shortfall/Conditional Value-at-Risk (ES). Tail-bounds lead to a infinite-dimensional Second Order Cone Program (SOCP) in occupation measures, while the ES approach creates a Linear Program (LP) in occupation measures. Under compactness and regularity conditions, there is no relaxation gap between the infinite-dimensional convex programs and their nonconvex optimal-stopping stochastic problems. Upper-bounds on the SOCP and LP are obtained by a sequence of semidefinite programs through the moment-Sum-of-Squares hierarchy. The VaR-upper-bounds are demonstrated on example continuous-time and discrete-time polynomial stochastic processes.

\end{abstract}


\section{Introduction}
\label{sec:introduction}

The behavior of stochastic processes can be interpreted by analyzing the time-evolving distributions of state functions $p(x)$ along trajectories. One such statistic is the $\epsilon$-\ac{VAR} of $p(x)$ with $\epsilon \in [0, 1)$, which also may be cast as the $(1-\epsilon)$-quantile statistic of $p(x)$ \cite{jorion2000value}. Our goal is to find the maximum \ac{VAR} obtained by $p(x)$ along trajectories of a stochastic process within a specified time horizon starting at a given initial condition. An example of this task in the context of aviation is to state that the supremal height of an aircraft with a $\geq 5\%$ chance of exceedence over the course of the flight is 100 meters.
We will refer to the task of upper-bounding the supremal \ac{VAR} along trajectories as the `chance-peak' problem.

The \ac{VAR} itself is a nonconvex and non-subadditive objective that is typically difficult to optimize. Two convex methods of upper-bounding for the \ac{VAR} are tail-bound and \ac{CVAR} approaches. Tail-bounds such as the Cantelli \cite{cantelli1929sui} and \ac{VP} \cite{vysochanskij1980justification} inequalities offer worst-case estimates on the \ac{VAR} of a univariate distribution given its first two moments \cite{dupacova1987minimax}. The \ac{CVAR} \cite{pflug2000some, rockafellar2002conditional} is a coherent risk measure \cite{artzner1999coherent} that returns the average value of a distribution conditioned upon being above the $\epsilon$-\ac{VAR}. The \ac{CVAR}  may be a desired optimization target \cite{sarykalin2008value} in terms of minimizing expected losses, given that the \ac{VAR} is invariant to the distribution shape past the $(1-\epsilon)$ quantile statistic.

Tail-bounds and \ac{CVAR} methods have both been previously applied to problems in control theory.
Tail-bound constraints have been employed for planning in \cite{wang2020non, han2022non}, in which moments of each time-step state-distribution are propagated using forward dynamics to the next discrete-time state.
\ac{CVAR} constraints for optimal control and safety have been used for continuous-time in \cite{cmiller2017optimal}, and for discrete-time Markov Decision Processes via  log-sum-exp \ac{CVAR}-upper-bounds in \cite{chapman2021risk}. We will be using the tail-bound and \ac{CVAR} upper-bounds to solve the chance-peak analysis problem.

The chance-peak problem is related to both chance constraints and peak estimation (optimal stopping). Chance constraints are hard optimization constraints that must hold with a specified probability \cite{shapiro2021lectures}.
Methods to convexly approximate these generically nonconvex chance constraints include tail-bounds, \ac{CVAR} programs, robust counterparts \cite{ben2009robust}, and scenario approaches (sampling)\cite{calafiore2005uncertain, campi2011sampling, tempo2013randomized}. The chance-constrained feasible set can be convex under specific distributional assumptions such as log-concavity \cite{lagoa2005probabilistically, ben2009robust}.

The chance-peak task is a specific form of an optimal stopping problem, in which the terminal time is chosen to maximize the \ac{VAR} of $p$. Optimal stopping is a specific instance of \iac{OCP}. The work in \cite{lewis1980relaxation} cast the generically nonconvex \ac{ODE} \acp{OCP} as a convex infinite-dimensional \ac{LP} in occupation measures with no conservatism (relaxation gap) introduced under continuity, compactness, and regularity assumptions. This method was extended in \cite{cho2002linear} to the optimal stopping of (Feller) stochastic processes such as \acp{SDE}, which formed an \ac{LP} to maximize the expectation of a state function $p$ along process trajectories.

These infinite-dimensional \acp{LP} must be truncated into finite-dimensional convex programs for tractable optimization. In the case where all problem data is polynomial, the moment-\ac{SOS} hierarchy of \acp{SDP} can be used to find a sequence of upper-bounds to the measure \acp{LP} \cite{lasserre2009moments}. Application of moment-\ac{SOS} polynomial optimization methods for deterministic or robust systems includes optimal control \cite{lasserre2008nonlinear},  reachable set estimation \cite{henrion2013convex}, and peak estimation \cite{fantuzzi2020bounding, miller2021uncertain}. Instances of polynomial optimization for stochastic processes include option pricing \cite{lasserre2006pricing}, probabilistic barrier certifiactes of safety \cite{prajna2004stochastic, prajna2007framework}, stopping of \levy processes \cite{kashima2010optimization}, infinite-time averages \cite{fantuzzi2016bounds}, and reach-avoid sets \cite{xue2022sdereachavoid}. We also note that polynomial optimization has been directly applied towards chance-constrained polynomial optimization \cite{jasour2015chance}, distributionally robust optimization \cite{de2020distributionally}, and minimization of the \ac{VAR} for static portfolio design \cite{tian2017moment}.

This paper puts forth the following contributions:
\begin{itemize}
    \item Infinite-dimensional convex programs (tail-bound \acp{SOCP} and \ac{CVAR} \acp{LP}) to upper-bound the \ac{VAR} of a stochastic process
    \item Converging sequences of \acp{SDP} in increasing size (degree) that upper-bound the infinite-dimensional program using the moment-\ac{SOS} hierarchy
    \item Experiments that demonstrate the utility of these methods on polynomial stochastic systems
\end{itemize}

Distinctions as compared to prior work include:
\begin{itemize}
    \item The chance-peak analysis problem maximizes the worst-case tail-bound/\ac{CVAR}, rather than minimizing worst-case upper-bounds on the \ac{VAR} (v.s. \cite{dupacova1987minimax, vcerbakova2006worst, shapiro2021lectures})
    \item Optimal stopping is performed with a tail-bound/\ac{CVAR} objective of $p$, rather than an expectation (mean) objective (v.s. \cite{cho2002linear, kashima2010optimization})
\end{itemize}

Sections of this research were accepted for presentation at the 62nd IEEE Conference on Decision and Control (CDC) \cite{miller2023chancepeak_short}. Contributions in this paper above the conference version include:
\begin{itemize}
    \item Non-\ac{SDE} stochastic processes
    \item \ac{CVAR} bounds and \acp{LP}    
    \item Safety analysis through distance estimation
    \item Proofs of no-relaxation-gap and strong duality
\end{itemize}

This paper is laid out as follows: Section \ref{sec:preliminaries} reviews notation, \ac{VAR} and its upper bounds, stochastic processes, and occupation measures. Section \ref{sec:problem_statement} poses the  chance-peak problem statement and lists relevant assumptions. Section \ref{sec:tail_bound} creates \iac{SOCP} in measures to bound the tail-bound chance-peak problem. Section \ref{sec:cvar_peak} formulates \iac{LP} in measures to solve the \ac{CVAR} chance-peak problem. Section \ref{sec:lmi} provides an overview of the moment-\ac{SOS} hierarchy of \acp{SDP}, and uses this hierarchy to approximate the tail-bound and \ac{CVAR} chance-peak programs. Section \ref{sec:extensions} extends the chance-peak framework towards the estimation of distance of closest approach to an unsafe set, and the analysis of switching stochastic processes. Section \ref{sec:examples} reports experiments of the tail-bound and \ac{CVAR} programs. Section \ref{sec:conclusion} summarizes and concludes the paper. Appendix \ref{app:duality_general} proves strong duality properties for a class of measure programs with linked semidefinite constraints. Appendix \ref{app:duality_chance} applies this general strong duality proof to the tail-bound chance-peak \acp{SOCP}. Appendix \ref{app:duality_cvar} proves strong duality of the \ac{CVAR} chance-peak \acp{LP}.
\section{Preliminaries}
\label{sec:preliminaries}


\begin{acronym}[SOCP]
\acro{BSA}{Basic Semialgebraic}


\acro{CVAR}[ES]{Expected Shortfall or Conditional Value-at-Risk}



\acro{LMI}{Linear Matrix Inequality}
\acroplural{LMI}[LMIs]{Linear Matrix Inequalities}
\acroindefinite{LMI}{an}{a}

\acro{LP}{Linear Program}
\acroindefinite{LP}{an}{a}

\acro{MC}{Monte Carlo}
\acroindefinite{MC}{an}{a}

\acro{OCP}{Ordinary Differential Equation}
\acroindefinite{OCP}{an}{a}

\acro{ODE}{Ordinary Differential Equation}
\acroindefinite{ODE}{an}{a}

\acro{PSD}{Positive Semidefinite}



\acro{RV}[RV]{Random Variable}

\acro{SDP}{Semidefinite Program}
\acroindefinite{SDP}{an}{a}

\acro{SDE}{Stochastic Differential Equation}
\acroindefinite{SDE}{an}{a}

\acro{SOC}{Second-Order Cone}
\acroindefinite{SOC}{an}{a}

\acro{SOCP}{Second-Order Cone Program}
\acroindefinite{SOCP}{an}{a}

\acro{SOS}{Sum of Squares}
\acroindefinite{SOS}{an}{a}

\acro{VAR}[VaR]{Value-at-Risk}

\acro{VP}{Vysochanskij-Petunin}


\end{acronym}

\subsection{Notation}



The real Euclidean space with $n$ dimensions is $\R^n$. The set of natural numbers is $\N$, the subset of natural numbers between $a$ and $b$ is $a..b \subset \N$, and the set of $n$-dimensional multi-indices is $\N^n$. The degree of a multi-index $\alpha \in \N^n$ is $\abs{\alpha} = \sum_{i=1}^n \alpha_i$. A monomial $x^\alpha = \prod_{i=1}^n x_i^{\alpha_i}$ has degree $\deg{x^\alpha} = \abs{\alpha} = \sum_{i=1}^n \alpha_i$. A polynomial $p(x) \in \R[x]$ is a linear combination of monomials
$p(x) = \sum_{\alpha \in S} p_\alpha x^\alpha$ {with finite support ${S \subset} \N^n$ and degree $\deg(p) = \max_{\alpha \in S}|\alpha|$}. The set $\R[x]_{\leq d}$of polynomials with degree at most $d$ forms a vector space of dimension is $\binom{n+d}{d}$. The $n$-dimensional \ac{SOC} is ${\mathbb{L}^n = \{(y, s) \in \R^n \times \R : \ s \geq \norm{y}_2\}}$, where $\norm{{y}}_2 = ({y}_1^2 + \ldots + {y}_n^2)^{1/2}$ is the Euclidean norm.

The topological dual of a Banach space ${B}$ is ${B}^*$. Given a topological space $X$, the set of all continuous functions over $X$ is $C(X)$, and the subcone of nonnegative continuous functions is $C_+(X) \subset C(X)$. The subset of $C(X)$ that is $k$-times continuously differentiable is $C^k(X)$. The cone of nonnegative Borel measures supported in $X$ is $\Mp{X}$. The space of signed Borel measures supported in $X$ is $\mathcal{M}(X) = \Mp{X} - \Mp{X}$. The sets $C(X)$ and $\mathcal{M}(X)$ are topological duals when $X$ is compact with a duality product $\inp{\cdot}{\cdot}$ via Lebesgue integration: for $f \in C(X), \ \mu \in \mathcal{M}(X), \inp{f}{\mu} = \int_{X} f(x) d \mu(x)$. The duality product over $C(X)$ and $\mathcal{M}(X)$ induces a duality pairing between $C_+(X)$ and $\mathcal{M}_+(X)$. We will slightly abuse notation to extend this duality product to Borel measurable functions $f$ as $\inp{f}{\mu} = \int_X f(x) d\mu(x)$.

The indicator function of $A \subseteq X$ is $I_A: X \rightarrow \{0, 1\}$ with value $1$ exactly on $A$. 
The measure of $A \subset X$ with respect to $\mu \in \Mp{X}$ is $\mu(A) = \inp{I_A}{\mu}$, and the mass of $\mu$ is $\mu(X) = \inp{1}{\mu}$. The measure $\mu$ is a probability measure if $\mu(X) = 1$, under which $(X,\mu)$ is a probability space). 

The pushforward of a measure $\mu \in \Mp{X}$ along a function ${V}: X \rightarrow Y$ is ${V}_\# \mu \in \Mp{Y}$ such that $\forall {\varphi} \in C(Y), \ \inp{{\varphi}}{{V}_\# \mu} = \inp{{\varphi\circ V}}{\mu}$. 

If moreover $\mu$ is a probability measure, then $V$ is called a \ac{RV} and one can define the expected value $\mathbb{E}[V] = \inp{V}{\mu}$ and probability $\mu(V \in A) = \inp{I_A}{V_\#\mu}$.
The support of a measure $\mu$ {(resp. \ac{RV} $V$)} is the set of all points $x$ in which every open neighborhood $N_x$ of $x$ obeys $\mu(N_x) > 0$ {(resp. $V_\#\mu(N_x)>0$)}.
The Dirac probability $\delta_{\overline{x}}$ supported at ${\overline{x} \in X}$ is such that $\inp{f}{\delta_{\overline{x}}} = f({\overline{x}})$ for all $f \in C(X)$. For every $\mu \in \Mp{X}$ and $\nu \in \Mp{Y}$, the product $\mu \otimes \nu$ is the unique measure satisfying $\forall A \subseteq X, \ B \subseteq Y: \ (\mu \otimes \nu)(A \times B) = \mu(A) \nu(B)$. 

The operator $\wedge$ {(resp. $\vee$)} will denote the {min. (resp. max.)} of two quantities as $a \wedge b = \min(a,b)$ {(resp. $a \vee b = \max(a,b)$)}. Given a linear operator $\Lie: X \rightarrow Y$, the adjoint of $\Lie: X \rightarrow Y$  is $\Lie^\dagger: Y^* \rightarrow X^*$.
The measure $\nu \in \Mp{X}$ is absolutely continuous to $\mu \in \Mp{X}$ ($\nu \ll \mu$) if $\forall A \subseteq X: \ \mu(A) = 0 \implies \nu(A) = 0$. {If} $\nu \ll \mu$, {then} there exists a nonnegative density $\rho: X \rightarrow \R^+$ such that $\forall f \in C(X): \inp{f(x)}{\nu(x)} = \inp{f(x)\rho(x)}{\mu(x)}$. This nonnegative density is also called the Radon-Nikodym derivative $\rho = \frac{d \nu}{d\mu}$. The measure $\mu$ dominates $\nu$ ($\nu \leq \mu$) if $\forall A \subseteq X: \nu(A) \leq \mu(A)$. Domination $\nu \leq \mu$ will occur if  $\nu \ll \mu$ and $\frac{d \nu}{d\mu} \leq 1$.

\subsection{Probability Tail Bounds and Value-at-Risk}
Define {a univariate \ac{RV} $V : (\Omega,\mathbb{P}) \rightarrow \mathbb{R}$ }
with finite first and second moments {($\abs{\mathbb{E}[V]} < \infty$, $\mathbb{E}[V^2] < \infty$, note that $V^\alpha$ is a \ac{RV} for any $\alpha \in \N$).} 
We will define the $\epsilon$-\ac{VAR} of {$V$} 
for $\epsilon \in [0, 1]$ as 
\begin{align}
    &\mathrm{VaR}_\epsilon(V) = \sup\left\{
    \lambda \in \R \; | \; \mathbb{P}(V \geq \lambda) \geq \epsilon
    \right\}. \label{eq:var_center}
\end{align}

Equation \eqref{eq:var_center} defines $\mathrm{VaR}_\epsilon({V})$ as the $(1-\epsilon)$ quantile statistic of ${V}$.
We now review two methods to upper-bound the \ac{VAR}: Concentration bounds and \ac{CVAR}.

\subsubsection{Concentration Bounds/Minimax}

This approach uses worst-case (minimax) bounds on the \ac{VAR} given the first and second moments of {$V$} \cite{dupacova1987minimax}. Letting 
{$\sigma^2 = \mathbb{E}[V^2] - \mathbb{E}[V]^2$} be the variance of {$V$}, the Cantelli \cite{cantelli1929sui} \ac{VAR} upper-bound is
\begin{subequations}
\label{eq:tail_bounds}
\begin{align}
    \mathrm{VaR}_\epsilon({V}) &\leq \sigma \sqrt{(1/\epsilon)- 1} + {\mathbb{E}[V] = \mathrm{VC}_\epsilon(V)}. \label{eq:var_cant}\\    \intertext{
    {When $V$} is unimodal and $\epsilon \leq 1/6$, the sharper \ac{VP} bound may be applied as in \cite{vysochanskij1980justification}}
    \mathrm{VaR}_\epsilon({V}) &\leq \sigma \sqrt{4/(9\epsilon) - 1} + {\mathbb{E}[V] = \mathrm{VP}_\epsilon(V)}. \label{eq:var_vp}
\end{align}
\end{subequations}

\subsubsection{Conditional Value-at-Risk / Expected Shortfall}

\begin{defn}
\label{defn:cvar}
The \ac{CVAR} is the mean value of ${V}$ such that ${V}$ is greater than or equal to the \ac{VAR} \cite[Equation 3]{rockafellar2002conditional}:
\begin{align}
    \mathrm{ES}_\epsilon(V) = \frac{1}{\epsilon} \, \mathbb{E}\left[I_{V \geq \mathrm{VaR}_\epsilon(V)} \, V\right]. \label{eq:cvar}
\end{align}

\end{defn}

\begin{rmk}
    A consequence of \eqref{eq:cvar} is {$\mathrm{ES}_\epsilon(V) \geq \mathrm{VaR}_\epsilon(V)$ for all \ac{RV} $V$} and probability values $\epsilon \in (0, 1]$. {Indeed, one has $\mathbb{E}[I_{V \geq \mathrm{VaR}_\epsilon(V)} \, V] \geq \mathrm{VaR}_\epsilon(V) \, \mathbb{E}[I_{V \geq \mathrm{VaR}_\epsilon(V)}] = \mathrm{VaR}_\epsilon(V) \, \mathbb{P}(V \geq \mathrm{VaR}_\epsilon(V)) \geq \epsilon \, \mathrm{VaR}_\epsilon(V)$ by definition of the \ac{VAR}.}
\end{rmk}

We now list other definitions for the \ac{CVAR}.

\begin{lem}[Equations 4 and 5 of \cite{rockafellar2002conditional}] {Defining the positive part $f_+ = f\vee 0$ of a function $f$, the} \ac{CVAR} is the solution to the parametric problem
\begin{align}
    \mathrm{ES}_\epsilon(V) = \min \left\{ \lambda + \frac{1}{\epsilon} \mathbb{E}\left[(V-\lambda)_+\right] \; \middle| \; \lambda \in \mathbb{R} \right\}. \label{eq:cvar_param}
\end{align}

\end{lem}

\begin{lem}[Equation 5.5 of \cite{follmer2010convex}] {Denoting $\psi = V_\#\mathbb{P}$ as the probability law of $V$ (equivalently, one can write $V \sim \psi$) and $\mathrm{id}_\mathbb{R} = \mathbb{R} \ni s \mapsto s$ as the identity map, the} \ac{CVAR} is the solution to the following optimization program in measures: \label{lem:cvar_abscont}
\begin{subequations}
\label{eq:cvar_abscont}
\begin{align}
    \mathrm{ES}_\epsilon(V) = & \sup_{\nu \in \Mp{\R}} \inp{\mathrm{id}_\mathbb{R}}{\nu} \label{eq:cvar_abscont_obj}\\
    & \nu \ll \psi \label{eq:cvar_abscont_def}\\
    & \frac{d\nu}{d\psi} \leq 1/\epsilon \label{eq:cvar_radon_nikodym} \\
    & \inp{1}{\nu} = 1. \label{eq:cvar_prob}
\end{align}

    \end{subequations}
\end{lem}

The objective in \eqref{eq:cvar_abscont_obj} is the mean of $\nu$.
Equation \eqref{eq:cvar_abscont_def} imposes that ${\nu}$ is absolutely continuous with respect to ${\psi}$. {From equation \eqref{eq:cvar_radon_nikodym}} ${\nu}$ possesses a Radon-Nikodym derivative  that has value $\leq 1/\epsilon$ at {any realization}. Equation \eqref{eq:cvar_prob} enforces that ${\nu}$ is a probability measure. 

\begin{rmk}
    The equation in \eqref{eq:cvar_abscont} is modified from \cite{follmer2010convex} to possess a supremization objective and to explicitly include the mass constraint \eqref{eq:cvar_prob}.
\end{rmk}

\begin{rmk}
When $\mathrm{VaR}_\epsilon({V})$ is not an atom of ${\psi}$, an analytical expression may be developed for ${\nu}$ solving \eqref{eq:cvar_abscont} as:
\begin{align}
    \nu^* = (1/\epsilon) \, I_{[\mathrm{VaR}_\epsilon(V),\infty)} \; \psi. \label{eq:nu_explicit}
\end{align}

Similar principles may be used to derive ${\nu^*}$ when ${\psi}$ has atomic components, but this may lead to splitting an atom.
\end{rmk}

Figure \ref{fig:es_bell} summarizes this subsection, with an example where $\psi$ is the unit normal distribution $V$ ($\mathbb{E}[V] = 0$, $\mathbb{E}[V^2] = 1$).  The blue curve is the probability density of $V$. The black area has a mass of $\epsilon = 0.1$, and the left edge of the black area is $\textrm{VaR}_{0.1}(V) = 1.2819$. The red curve is $\epsilon \nu^* = I_{[\mathrm{VaR}_\epsilon(V),\infty)} \; \psi$ from \eqref{eq:nu_explicit}. The green dotted line is the $\textrm{ES}_{0.1}(V) = 1.7550$. The \ac{VP} and Cantelli bounds are $\textrm{VP}_{0.1}(V)=1.8559$ and $\textrm{VC}_{0.1}(V) = 3$ respectively.
\begin{figure}[h]
    \centering
    \includegraphics[width=0.7\linewidth]{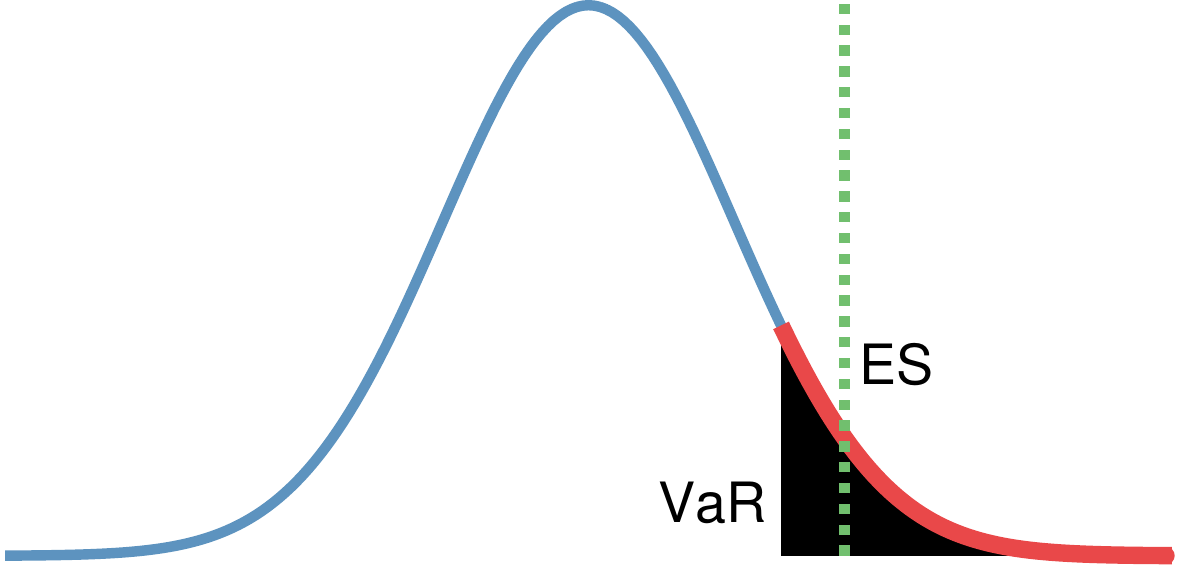}
    \caption{\ac{VAR} and \ac{CVAR} of a unit normal distribution at $\epsilon = 0.1$}
    \label{fig:es_bell}
\end{figure}

\subsection{Stochastic Processes and Occupation Measures}
\label{sec:prelim_sde}

Let $\{\mu_t\} \in \Mp{X}$ be a time-indexed sequence of probability distributions. Define $\mathcal{T}_\tau$ as a time-shifting (Feller) semigroup operator acting as $\mathcal{T}_\tau \mu_t= \mu_{t+\tau}$. The \textit{generator} $\Lie_\tau$ of a stochastic processes associated with the distributions $\{\mu_t\}$ is a linear operator satisfying (for any test function $v(t, x) \in C([0, T] \times X)$ and in the domain of $\Lie$):
\begin{equation}
    \Lie_\tau v= \lim_{\tau' \rightarrow \tau} (\inp{v(t+\tau', x)}{\mu_{t+\tau'}}- \inp{v(t,x)}{\mu_t})/\tau'. \label{eq:generator_lim}
\end{equation}

A discrete-time Markov stochastic process with parameter distribution $\xi(\lambda) \in \Mp{\Lambda}$ and time-step $\Delta t> 0$ has a law and associated generator of 
\begin{align}
    x[t+\Delta t] &=  f(t, x[t], \lambda[t]), \qquad \lambda[t] \sim \xi \\
    \Lie_{\Delta t} v &= \left(\int_{\Lambda}v(t+\Delta t, f(t, x, \lambda)) d \xi(\lambda) - v(t, x)\right)/\Delta t. \label{eq:generator_discrete}
\end{align}
The domain of $\Lie_{\Delta t}$ in \eqref{eq:generator_discrete} is $\cs = C([0, T] \times X)$.

\ito \acp{SDE} are the unique class of continuous-time stochastic processes that are nonanticipative, have independent increments, and possess continuous sample paths \cite{oksendal2003stochastic}. Every \ito \ac{SDE} has a definition in terms of a (possibly nonunique) drift function $f$, a diffusion function $g$, and an $n$-dimensional Wiener process $W$ as 
\begin{equation}
\label{eq:sde}
    dx = f(t, x) dt + g(t, x) d{W}.
\end{equation}

This paper will involve stochastic trajectories evolving in a compact set $X$, starting within an initial set $X_0$ at time $0$. Letting $\tau_X$ be a stopping time (\ac{RV}) associated with the first touch of the boundary $\partial X$, the strong \ac{SDE} solution of \eqref{eq:sde} in times $t \in [0, T]$ starting at an initial point $x(0)\in X_0$ is
\begin{equation}
\label{eq:sde_sol}
    x(t) = x(0) + \int_{t=0}^{\tau_X \wedge T} f(t, x) dt + \int_{t=0}^{\tau_X \wedge T} g(t, x) d{W}.
\end{equation}

Strong solutions of \eqref{eq:sde_sol} (sequence of state-probability-distributions $\{\mu_t\}$ describing $x(t)$) are unique if there exists $C, D>0$ such that the following Lipschitz and Growth conditions hold for all $(t, x, x') \in [0, T] \times X^2$ \cite{oksendal2003stochastic}:
 \begin{align}
    D \norm{x-x'}_2 &\geq \norm{f(t,x)-f(t,x')}_2 + \norm{g(t, x)-g(t,x')}_2  \nonumber \\
     C (1+\norm{x}_2) & \geq \norm{f(t,x)}_2 + \norm{g(t, x)}_2. \label{eq:lip_growth}
 \end{align}
These Lipschitz and Growth conditions will be satisfied if $X$ is compact and $(f, g)$ are locally Lipschitz. The generator associated to \eqref{eq:sde} with domain $C^{1, 2}([0, T] \times X)$ is
\begin{equation}
\label{eq:lie}
    \Lie v(t, x) = \partial_t v + f(t, x) \cdot \nabla_x v + \frac{1}{2} g(t, x)^T \left(\nabla^2_{xx}v\right)  g(t, x).
\end{equation}
In the rest of this paper, we denote the domain of the generator $\Lie$ by $\cs$, and will refer to stochastic processes by their generators ($x(t)$ or $\{\mu_t\}$ satisfy the stochastic process of $\Lie$ in time $T \wedge \tau_X$).

Let $\mu_0 \in \Mp{X_0}$ be an initial distribution, $t' \in [0, T]$ be a terminal time, and $\tau = t' \wedge \tau_X$ be an associated stopping time. The occupation measure $\mu \in \Mp{[0, T] \times X}$ and stopping measure $\mu_\tau \in \Mp{[0, T] \times X}$ of the stochastic processes $\Lie$ w.r.t. $\mu_0$ is $\forall A \subseteq [0, T], B \subseteq X$:
\begin{subequations}
\begin{align}
    \label{eq:avg_free_occ}
    \mu(A \times B) &= \int_{X_0} \int_{t=0}^\tau  I_{A \times B}\left(t, x(t \mid x_0)\right) dt \, d\mu_0(x_0) \\
    \label{eq:avg_free_stop}
    \mu_\tau(A \times B) &= \int_{X_0} I_{A \times B}\left(\tau, x(\tau \mid x_0)\right) \, d\mu_0(x_0).
\end{align}
\end{subequations}

The measures $(\mu_0, \mu, \mu_\tau)$ together obey a martingale relation \cite{rogers2000diffusions}:
\begin{align}
\label{eq:dynkin_strong}
    \inp{v}{\mu_\tau} &= \inp{v(0, x)}{\mu_0(x)} + \inp{\Lie v}{\mu} & & \forall v \in \cs,
\end{align}
which can be equivalently expressed in shorthand form as
\begin{align}
    \mu_\tau &= \delta_{0} \otimes \mu_0 + \Lie^\dagger \mu. \label{eq:dynkin}
\end{align}

The martingale relation in \eqref{eq:dynkin} is known as Dynkin's formula \cite{dynkin1965markov} when $\Lie$ is the generator for \iac{SDE}. A \textit{relaxed occupation measure} is a tuple $(\mu_0, \mu, \mu_\tau)$ satisfying \eqref{eq:dynkin} with $\inp{1}{\mu_0} = 1$.

An optimal stopping problem to maximize the expectation of the reward $p(x)$ along trajectories of $\Lie$ is $P^* = \sup_{t' \in [0, T]} \inp{p(x)}{\mu_{t'}(x)}$. 
This stopping problem may be upper-bounded by an infinite-dimensional \ac{LP} in measures
\begin{subequations}
\label{eq:peak_meas}
\begin{align}
p^* = & \ \sup \quad \inp{p}{\mu_\tau} \label{eq:peak_meas_obj} \\
    & \mu_\tau = \delta_0 \otimes\mu_0 + \Lie^\dagger \mu \label{eq:peak_meas_flow}\\
    & \inp{1}{\mu_0} = 1 \label{eq:peak_meas_prob}\\
    & \mu, \mu_\tau \in \Mp{[0, T] \times X}.  \notag
\end{align}
\end{subequations}

The upper-bound is tight $(p^* = P^*)$ when $p$ is continuous, $[0, T] \times X$ is compact, and closure conditions hold on the generator $\Lie$ \cite{cho2002linear} (to be reviewed in Section \eqref{sec:assum}). \ac{SDE} trajectories under Lipschitz and Growth \eqref{eq:lip_growth} conditions will satisfy these requirements.

\section{Chance-Peak Problem Statement}
\label{sec:problem_statement}

This section presents the problem statement for the chance-peak problem, and will formulate the tail-bound and \ac{CVAR} upper-bounding programs to peak-\ac{VAR} estimation.




\subsection{Assumptions}
\label{sec:assum}

We will require the following assumptions:
\begin{itemize}
    \item[A1] The spaces $[0, T]$ and $X$ are compact.
    \item[A2] Trajectories stop upon their first contact with $\partial X$.
    \item[A3] The state function $p$ is continuous on $X$.
    \item[A4] The initial measure $\mu_0$ satisfies $\inp{1}{\mu_0} = 1$ and $\supp{\mu_0} \subseteq X$.
    \item[A5] The set $\cs = \textrm{dom}(\Lie) \subset \textrm{Hom}(C([0, T] \times X), C([0, T] \times X))$ contains $1 \in \cs$ with $\Lie 1 = 0$.
    \item[A6] The set $\cs$ can separate points and is closed under multiplication.
    \item[A7] There exists a countable basis  $\{v_k\} \in \cs$ such that every $(v, Av)$ with $v \in \cs$ is contained in the (bounded pointwise closure of the) linear span of $\{(v_k, \Lie v_k)\}$.
\end{itemize}

Assumptions A5-A7 originate from the requirements of Condition 1 of \cite{cho2002linear}. Discrete-time Markov processes, \acp{SDE}, and L\'{e}vy processes in compact domains (A1, A2) will satisfy conditions A5-A7 \cite{cho2002linear}. Assumption A5 and \eqref{eq:generator_lim} together imply that $\Lie t = 1$.



\subsection{VaR Problem}
\begin{prob}
\label{prob:quantile}
The chance-peak problem that maximizes the $\epsilon$-\ac{VAR} of $p$ is

    
\begin{subequations}
\label{eq:peak_chance}
\begin{align}
P^* = & \sup_{t^* \in [0, T]}  \label{eq:peak_chance_obj} \ \mathrm{VaR}_{\epsilon}(p{(x(t^*))}) \\
     & \text{$x(t)$ follows $\Lie$ from $t=0$ until } \ \tau_X \wedge t^* \label{eq:peak_chance_stop}\\
     & x(0) \sim \mu_0.
\end{align}
\end{subequations}
\end{prob}

\section{Tail-Bound Program}
\label{sec:tail_bound}

This section solves the chance-peak problem \ref{prob:quantile} using tail-bounds by developing an infinite-dimensional \ac{SOCP} in measures.

\subsection{Tail-Bound Problem}


We define $r$ as the constant factor multiplying against $\sigma$ in the Cantelli or \ac{VP} expression of \eqref{eq:tail_bounds} as in
\begin{align}
\label{eq:tail_bounds_constant}
    {r_C} &= \sqrt{(1/\epsilon)-1} & {r_{VP}} &=  \sqrt{4/(9\epsilon)-1},
\end{align}
ensuring that the \ac{VP} bound is used if and only if its $\epsilon \leq 1/6$ and unimodality conditions are satisfied. We also use the notation $\inp{p^2}{\mu_{t^*}}$ to denote the expectation {$\mathbb{E}[p(x(t^*))^2]$}.
\begin{prob}
    The tail-bound program upper-bounding \eqref{eq:peak_chance} with constant $r$ is
\begin{subequations}
\label{eq:peak_chance_tail}
\begin{align}
    P^*_r = & \sup_{t^* \in [0, T]} r\sqrt{\inp{p^2}{\mu_{t^*}} - \inp{p}{\mu_{t^*}}^2}  + \inp{p}{\mu_{t^*}} \label{eq:peak_chance_tail_obj}\\     
    & \text{$x(t)$ follows $\Lie$ from $t=0$ until } \ \tau_X \wedge t^* \label{eq:peak_chance_tail_stop}\\
     & x(0) \sim \mu_0.
\end{align}
\end{subequations}
\end{prob}

\subsection{Nonlinear Measure Program}

Problem \eqref{eq:peak_chance_tail} can be converted into an infinite-dimensional nonlinear program in variables $(\mu_\tau, \mu)$, given the initial distribution $\mu_0$ and the generator $\Lie$:
\begin{subequations}
\label{eq:peak_chance_meas}
\begin{align}
    p^*_r = & \sup r\sqrt{\inp{p^2}{\mu_{\tau}} - \inp{p}{\mu_{\tau}}^2}  + \inp{p}{\mu_{\tau}} \label{eq:peak_chance_meas_obj}\\
     & \mu_\tau = \delta_0 \otimes \mu_0 + \Lie^\dagger \mu \label{eq:peak_chance_meas_lie}\\
    & \mu_\tau, \ \mu \in \Mp{[0, T] \times X}. \notag
\end{align}
\end{subequations}

\begin{thm}
\label{thm:upper_bound_nonlinear}
Programs \eqref{eq:peak_chance_meas} and \eqref{eq:peak_chance_tail} are related by $p^*_r \geq P^*_r$ under assumptions A1, A2, A4.
\end{thm}
\begin{proof}
Let $t^* \in [0, T]$ be a terminal time with stopping time of $\tau^* = t^* \wedge \tau_X$. We can construct measures $(\mu, \mu_\tau)$ to satisfy \eqref{eq:peak_chance_meas_lie} from the data in $(t^*, \Lie, \mu_0)$. Specifically, let $\mu$ be the occupation measure of the stochastic process $\Lie$ starting from $\mu_0$ until time $\tau^*$, and let $\mu_\tau$ be the time-$\tau^*$ state distribution. The allowed stopping times $t^* \in [0, T]$ for \eqref{eq:peak_chance_tail} maps in a one-to-one manner on the feasible set of constraints \eqref{eq:peak_chance_meas}, affirming that $p^*_r \geq P^*_r$.
\end{proof}

\begin{rmk}
A worst-case tail-bound $p^*_r$ for \eqref{eq:peak_chance_meas} can be found from an initial set $X_0 \subseteq X$ by letting the initial condition $\mu_0$ be  an optimization variable with the constraints $\mu_0 \in \Mp{X_0}$ and $\inp{1}{\mu_0} = 1$.
\end{rmk}

\begin{lem}
\label{lem:no_relaxation_gap}
Under assumptions (A1, A2, A4-7), to every $(\mu_\tau, \mu)$ obeying \eqref{eq:var_meas_lie_soc}  there exists a stochastic process trajectory \eqref{eq:peak_chance_tail_stop} starting at $\mu_0$ and terminating at $\mu_\tau$ with an occupation measure of $\mu$.
\end{lem}
\begin{proof}
    This lemma holds by Theorem 3.3 of \cite{cho2002linear} (under A1, A2, A5-7; noting that $[0, T] \times X$ is a complete metric space).
\end{proof}

\begin{cor}
    The combination of Theorem \ref{thm:upper_bound_nonlinear} and Lemma \ref{lem:no_relaxation_gap} implies that $p^*_r = P^*_r$ under assumptions A1-A7.
\end{cor}

\subsection{Measure Second-Order Cone Program}

This subsection will demonstrate how the  nonlinear program in measures \eqref{eq:peak_chance_meas} may be equivalently expressed as an infinite-dimensional convex measure \ac{SOCP}. We will use an \ac{SOC} representation of the square-root to accomplish this conversion through the following lemma:

\begin{lem}
\label{lem:sqrt}
Define the objective \eqref{eq:peak_chance_meas_obj} as $J_r(a, b) = r \sqrt{b - a^2} + a$ under the substitutions $a=\inp{p}{\mu_\tau}$ and $b = \inp{p^2}{\mu_\tau}$. Given any convex set 
${K} \in \R \times \R_+$ with $(a, b) \in {K}$, the subsequent programs will possess the same optimal value:
\begin{align}
    &\sup_{(a, b) \in {K}} a + r \sqrt{b-a^2}  \\
    &\sup_{(a, b) \in {K}, \ {c} \in \R} a + r { \, c}: \ ([1-b, 2{c}, 2a],  1+b) \in {\mathbb{L}}^3. \label{eq:square_root_socp}
\end{align}

\end{lem}
\begin{proof} We introduce a new variable $c$ as in $c \leq \sqrt{b-a^2}$, which implies that ${c}^2 + a^2 \leq b $. The \ac{SOC}-equivalent form of $\sqrt{b-a^2}$ follows from \cite{alizadeh2003second, yalmip2009sqrt},
\begin{align*}
    &([1 - b, 2{c}, 2a],  1+b) \in {\mathbb{L}}^3  \\
    \Longleftrightarrow \; & (1-b)^2 + 4({c}^2 + a^2) \leq (1+b)^2 \\
    \Longleftrightarrow \; & (1+b^2) - 2b + 4({c}^2 + a^2) \leq (1+b^2) + 2b \\
    \Longleftrightarrow \; & 4({c}^2 + a^2) \leq 4b.
\end{align*}
\end{proof}

\begin{thm}
\label{thm:upper_bound_socp}
The nonlinear program \eqref{eq:peak_chance_meas} has the same set of feasible solutions and  optimal value as the following \ac{SOCP} (given $(\mu_0, \Lie, r)$):
\begin{subequations}
\label{eq:var_meas_soc}
\begin{align}
    p^*_{r} = &\sup \quad r {c}  + \inp{p}{\mu_\tau} \label{eq:var_meas_obj_soc}\\
     & \mu_\tau = \delta_0 \otimes \mu_0 + \Lie^\dagger \mu \label{eq:var_meas_lie_soc}\\
    & {y} = [1-\inp{p^2}{\mu_\tau}, \ 2 {c}, \ 2 \inp{p}{\mu_\tau}] \label{eq:var_meas_con_soc_def}\\
    & ({y}, 1 + \inp{p^2}{\mu_\tau}) \in {\mathbb{L}}^3 \label{eq:var_meas_con_soc}\\
    & \mu, \ \mu_\tau \in \Mp{[0, T] \times X}, u \in \R, c \in \R^3.\notag
\end{align}
\end{subequations}
\end{thm}
\begin{proof}
This equivalence follows from Lemma \ref{lem:sqrt} by replacing the square-root in the objective \eqref{eq:peak_chance_meas_obj}. The new optimization variables are $(\mu_\tau, \mu, u, c)$.
\end{proof}

\begin{cor}
The \ac{SOCP} in \eqref{eq:var_meas_soc} is convex.
\end{cor}
\begin{proof}
Constraints \eqref{eq:var_meas_lie_soc}-\eqref{eq:var_meas_con_soc} are convex (\ac{SOC} for \eqref{eq:var_meas_con_soc} and affine for \eqref{eq:var_meas_lie_soc}). The objective  \eqref{eq:var_meas_obj_soc} is linear in $({c}, \mu_\tau)$.
\end{proof}





\subsection{Dual Second-Order Cone Program}

The Lagrangian dual of \eqref{eq:var_meas_soc} is a program with infinite-dimensional linear constraints and a finite-dimensional SOC constraint. This dual involves a function $v(t, x) \in \cs([0, T] \times X)$ and a constant $u \in \R^3$ as variables.

We will use the following expression of \eqref{eq:var_meas_soc} with an explicitly written \ac{SOC} variable $z$ linked with linear constraints to the measures $(\mu_0, \mu_\tau, \mu)$. 
\begin{lem}
The following program has the same optimal value as \eqref{eq:var_meas_soc}:
\begin{subequations}
\label{eq:var_meas_soc_ext}
\begin{align}
    p^*_{r} = &\sup \quad (r/2) z_2  + \inp{p}{\mu_\tau} \label{eq:soc_cost} \\
    & z_1+\inp{p^2}{\mu_\tau}= 1 \label{eq:soc_begin}\\
    & z_3 - 2 \inp{p}{\mu_\tau} =0 \\
    & z_1 + z_4= 2 \label{eq:soc_end_aff}\\
     & \mu_\tau - \Lie^\dagger \mu = \delta_0 \otimes \mu_0 \label{eq:soc_liou} \\
    & z =([z_1, z_2, z_3], z_4) \in \mathbb{L}^3 \notag\\
    & \mu, \ \mu_\tau \in \Mp{[0, T] \times X}  \notag.
\end{align}
\end{subequations}
\end{lem}
\begin{proof}
This formulation is obtained from \eqref{eq:var_meas_soc} through the change of variable $z = (y,1+\inp{p^2}{\mu_\tau})$ and the replacement of $c$ with $z_2/2$ using the second coordinate of constraint \eqref{eq:var_meas_con_soc_def}. Equation \eqref{eq:soc_end_aff} is derived by adding the first coordinate of \eqref{eq:var_meas_con_soc_def} and the last coordinate of \eqref{eq:var_meas_con_soc}.

\end{proof}

\begin{thm} \label{thm:cdc_strong_dual}
The dual program of \eqref{eq:var_meas_soc} with weak duality $d^*_r \geq p^*_r$ under Assumptions A1-A3 is
\begin{subequations}
\label{eq:var_cont_soc}
\begin{align}
    d^*_{r} =& \inf \quad u_1 + 2 u_3 + \int_{X_0} v(0, x_0) d\mu_0(x_0) \label{eq:lag_cost} \\
    &\forall (t, x) \in [0, T] \times X: \nonumber \\
    & \qquad \Lie v(t,x) \leq 0 & & \label{eq:lag_occ_meas} \\
    & \forall (t, x) \in [0, T] \times X: \nonumber \\
    & \qquad v(t, x) + u_1 \, p^2(x) - 2 \, u_2 \, p(x) \geq p(x)  \label{eq:lag_stop_meas} \\
    & ([u_1+u_3, -(r/2), u_2], u_3) \in {\mathbb{L}}^3 \label{eq:lag_soc} \\
    & u \in \R^3, \ v \in \cs([0, T] \times X). \notag
\end{align}
\end{subequations}
Strong duality with $d^*_r = p^*_r$ holds under Assumptions A1-A4.
\end{thm}
\begin{proof}
\textbf{Dual formulation:} this formulation is obtained by applying the standard Lagrangian duality method to \eqref{eq:var_meas_soc_ext}. $v$ is the Lagrange multiplier corresponding to constraint \eqref{eq:soc_liou}, and $u$ is the Lagrange multiplier corresponding to constraints \eqref{eq:soc_begin}-\eqref{eq:soc_end_aff}. Conversely, $\mu$ is the Lagrange multiplier corresponding to constraint \eqref{eq:lag_occ_meas}, $\mu_\tau$ is the Lagrange multiplier corresponding to constraint \eqref{eq:lag_stop_meas}, and $z$ is the Lagrange multiplier corresponding to \eqref{eq:lag_soc}. The cost in \eqref{eq:lag_cost} corresponds to the right-hand sides of constraints \eqref{eq:soc_begin}-\eqref{eq:soc_liou}, while the right-hand side of \eqref{eq:lag_stop_meas} and the second coordinate $-(r/2)$ in \eqref{eq:lag_soc} correspond to the cost in \eqref{eq:soc_cost}.

\textbf{Strong duality:} see Appendix \ref{app:duality_chance}.
\end{proof}

This strong duality property is an important feature of the infinite-dimensional problem at hand: it means that one may equivalently solve moment relaxations of \eqref{eq:var_meas_soc_ext} and \ac{SOS} tightenings of \eqref{eq:var_cont_soc}.
 
\section{ES Program}

\label{sec:cvar_peak}

This section poses \iac{LP} in measures to solve the \ac{CVAR} upper-bound to  Problem \ref{prob:quantile}.

\subsection{ES Problem}
The \ac{CVAR} chance-peak problem replaces the \ac{VAR} objective in  \eqref{eq:peak_chance_obj} with \ac{CVAR}.
\begin{prob}
\label{prob:cvar_peak}
    The \ac{CVAR} program that upper-bounds the chance-peak problem in \eqref{eq:peak_chance} is
    \begin{subequations}
\label{eq:peak_cvar}
\begin{align}
    P^*_c = & \sup_{t^* \in [0, T] }  \mathrm{ES}_{\epsilon}(p(x(t^*))) \label{eq:peak_cvar_obj}\\
     & \text{$x(t)$ follows $\Lie$ from $t=0$ until } \ \tau_X \wedge t^* \label{eq:peak_cvar_stop}\\
     & x(0) \sim \mu_0.
\end{align}
\end{subequations}
\end{prob}


\subsection{ES reformulation}

We begin by using the following lemma to reformulate the absolute-continuity-based \ac{CVAR} definition in \eqref{eq:cvar_abscont} into an equivalent domination-based \ac{LP} in measures.

\begin{lem}
\label{lem:abscont_domination}
    Let ${\mu}, \nu \in \Mp{\R}$ be measures {such that $\nu \ll \mu$}, and {$\frac{d \nu}{ d \mu} \leq 1$}. Then there exists a slack measure ${\hat{\nu}} \in \Mp{\R}$ such that ${\nu + \hat{\nu} = \mu}$. 
\end{lem}
\begin{proof}
    The measure ${\hat{\nu}}$ may be chosen with {$\frac{d \hat{\nu}} {d \mu} = 1 - \frac{d\nu} {d \mu}$}. This implies that ${\frac{d \nu} {d \mu} + \frac{d \hat{\nu}} {d \mu}} = 1$, resulting in {$\nu + \hat{\nu} = \mu$}.
\end{proof}

\begin{thm}
\label{thm:cvar_dom}
The \ac{CVAR} is the solution to the following \ac{LP} in measures:
\begin{subequations}
\label{eq:cvar_dom}
\begin{align}
    \mathrm{ES}_\epsilon(V) = & \sup_{\nu, \hat{\nu} \in \Mp{\R}} \inp{\mathrm{id}_\R}{\nu} \\
    & \epsilon\nu + \hat{\nu} = \psi = V_\#\mathbb{P} \label{eq:cvar_dom_cond}\\
    & \inp{1}{\nu} = 1. \label{eq:cvar_mass}
\end{align}
    \end{subequations}
\end{thm}
\begin{proof} 
Equation \eqref{eq:cvar_radon_nikodym} may be divided through by $\epsilon$ to form $\epsilon {\frac{d \nu}{ d \psi} =  \frac{d (\nu \epsilon)}{ d \psi}} \leq 1$. Given that $\epsilon > 0$, the absolute continuity relation ${\nu \ll \psi}$ implies that $\epsilon {\nu \ll \psi}$. Lemma \ref{lem:abscont_domination} is applied to the combination of $
    \epsilon {\nu \ll \psi}$ and $ {\frac{d (\nu \epsilon)}{ d \psi}} \leq 1$ to produce constraint \eqref{eq:cvar_dom_cond}. This proves the conversion and equivalence of optima between \eqref{eq:cvar_abscont} and \eqref{eq:cvar_dom}.



\end{proof}

\subsection{Measure Program}


\Iac{LP} in measures will be created to upper-bound the \ac{CVAR} program \eqref{eq:peak_cvar}. The variables involved are the terminal measure $\mu_\tau$, the relaxed occupation measure $\mu$, the \ac{CVAR} dominated measure ${\nu}$, and the \ac{CVAR} slack measure ${\hat{\nu}}$.

\begin{subequations}
\label{eq:peak_cvar_meas}
{
\begin{align}
p^*_c = & \ \sup \quad \inp{\mathrm{id}_\R}{\nu} \label{eq:peak_cvar_meas_obj} \\
    & \mu_\tau = \delta_0 \otimes\mu_0 + \Lie^\dagger \mu \label{eq:peak_cvar_meas_flow}\\
    & \inp{1}{\nu} = 1\label{eq:peak_cvar_meas_prob}\\
    & \epsilon \nu + \hat{\nu} = p_\# \mu_\tau \label{eq:peak_cvar_meas_cvar}\\
    & \mu, \mu_\tau \in \Mp{[0, T] \times X} \notag\\    
    & \nu, \hat{\nu} \in \Mp{\R}. \notag
\end{align}
}
\end{subequations}

\begin{thm}
\label{thm:upper_bound_cvar}
    Program \eqref{eq:peak_cvar_meas} upper-bounds \eqref{eq:peak_cvar} with $p^* \geq P^*$ under assumptions A2-A4.
\end{thm}
\begin{proof}
    We will prove this upper-bound by constructing a measure representation of an \ac{SDE} trajectory in \eqref{eq:peak_cvar}.
    Let $t^* \in [0, T]$ be a stopping time. Define $\mu_\tau = \mu_{t^*}$ as the state (probability) distribution of the process \eqref{eq:peak_cvar_stop} at time $t^*$ (accounting for the stopping time $\tau_X \wedge t^*$). Let $\mu$ be the occupation measure of this \ac{SDE} connecting together the initial distribution $\mu_0$ and the terminal distribution $\mu_{t^*}$. The measures ${\nu, \hat{\nu}}$ are set in accordance with Theorem \ref{thm:cvar_dom} under { $\Omega = [0,T]\times X$, $\mathbb{P} = \mu_{t^*}$ and $V=p$ (hence} ${\psi} = p_\# \mu_{t^*}$). The upper-bound holds because every process trajectory has a measure construction.
\end{proof}

\begin{thm}
\label{thm:no_relaxation_cvar}
    There is no relaxation gap between \eqref{eq:peak_cvar_meas}  and \eqref{eq:peak_cvar} $(p^*_c = P^*_c)$ under assumptions A1-A7. 
\end{thm}
\begin{proof}
Every $(\mu_\tau, \mu)$ is supported on an stochastic process trajectory by Lemma \ref{lem:no_relaxation_gap}.
The objective in \eqref{eq:peak_cvar_meas_obj} will equal the \ac{CVAR} by Theorem \ref{thm:cvar_dom} given that $p_\# \mu_\tau$ is a probability distribution. 
\end{proof}

\begin{rmk}
To the best of our knowledge, there does not appear to be a consistent comparison between the \ac{VP} bound and the \ac{CVAR} of a unimodal distribution. 
\end{rmk}

\subsection{Function Program}

The functional \ac{LP} dual to \eqref{eq:peak_cvar_meas} will have variables {$u \in \R$,} $v \in \cs([0,T]\times X)$ and ${w} \in C(\R)$ (dual to \eqref{eq:peak_cvar_meas_flow}-\eqref{eq:peak_cvar_meas_cvar}). 
\begin{thm} \label{thm:cvar_strong_dual}
The strong-dual program of \eqref{eq:peak_cvar_meas} with duality $d^*_c {=} p^*_c$ under A1-A4 is
\begin{subequations}
\label{eq:peak_cvar_cont}
\begin{align}
    d^*_{c} =& \inf \quad  {u + \inp{v(0,\bullet)}{\mu_0}}\\ 
    & \Lie v \leq 0 \label{eq:lag_occ_meas_cvar}\\
    & v \geq {w \, \circ \, }p \label{eq:lag_stop_meas_cvar} \\
    & {u \ + \ } \epsilon \, {w} \geq \mathrm{id}_\R \\
    & {w} \geq 0 \\
    & {u \in \R}, v \in \cs{, w \in C(\R)}.
\end{align}
\end{subequations}
\end{thm}
\begin{proof}
    See Appendix \ref{app:duality_cvar}.
\end{proof}

\section{Finite Moment Program}
\label{sec:lmi}

This section will apply the moment-\ac{SOS} hierarchy of \acp{SDP} to develop upper-bounds of \eqref{eq:var_meas_soc} and \eqref{eq:peak_cvar_meas}.



\subsection{Review of Moment-SOS Hierarchy}

Refer to \cite{lasserre2009moments} for a more complete introduction to concepts reviewed in this subsection.
Letting $\bm = \{\bm_\alpha\}_{\alpha\in\N^n} \in \R^{\N^n}$ be a multi-indexed vector, we can define a Riesz linear functional  $L_\bm: \ \R[x] \rightarrow \R$ as
\begin{subequations}
\begin{equation}
\label{eq:riesz}
    \begin{array}{ccc}
      p(x) = \sum\limits_{\alpha\in\N^n} p_\alpha x^\alpha & \longmapsto & \textstyle L_\bm p = \sum\limits_{\alpha\in \N^n} p_\alpha \bm_\alpha.
    \end{array}
\end{equation}

For any measure $\mu \in \Mp{X}$ and multi-index $\alpha \in \N^n$, the $\alpha$-moment of $\mu$ is $\bm_\alpha = \inp{x^\alpha}{\mu}$. The infinite-dimensional moment sequence $\bm = \{\bm_\alpha\}_{\alpha \in \N^n}$ is related to the linear functional $L_\bm$ by
\begin{equation}
    \label{eq:momriesz}
    \forall p \in \R[x], \quad L_\bm p = \inp{p}{\mu}.
\end{equation}
\end{subequations}




Given a moment sequence $\bm \in \R^{\N^n}$  and a polynomial $h \in \R[x]$, a localizing bilinear functional can be defined by $L_{h \bm}: \R[x] \times \R[x] \longrightarrow \R$ by
\begin{equation}
    L_{h\bm} = (p,q) \longmapsto L_\bm(hpq).
\end{equation}

The cone of polynomials $\R[x]$ can be treated as a vector space with a linear basis $(e_i)_{i\in\N}$ (e.g. $e_i(x) = x^{\alpha_i}$ with $\{\alpha_i\}_{i\in\N} = \N^n$ and a monomial ordering $|\alpha_i| < |\alpha_j| \Rightarrow i<j$). The bilinear functional $L_{h \bm}$ has a representation as a quadratic form operating on an infinite-size localizing matrix $\M[h\bm]$ with $\M[h\bm] = (L_\bm(h \, e_i \, e_j))_{i,j \in \N}$. The expression for $\M[h\bm]$ when $\{e_j\}$ is the set of ordered monomials leads to
\begin{equation}
    \label{eq:repmat}
    \M[h\bm]_{i,j} 
    = L_\bm(h x^{\alpha_i} x^{\alpha_j})
    = \sum_{\beta\in\N^n} h_\beta \bm_{\alpha_i+\alpha_j+\beta}.
\end{equation}


\Iac{BSA} set $\K = \{x \mid h_k(x) \geq 0: \ k = 1..N_c\}$ is a set defined by a finite number of bounded-degree polynomial inequality constraints $h_k(x) \geq 0$.
We refer to any set $\K$ that has $h_1(x) = 1$ and $h_{N_c} = R - \norm{x}_2^2$ with $R>0$ as satisfying `ball constraints,' and note that such an $R>0$ can be chosen for any compact set. 
For any set obeying ball constraints, then there exists a representing measure $\mu$ for the sequence of numbers (pseudomoments) $\bm$ if each localizing matrix is \ac{PSD}:
\begin{equation}
    \label{eq:truncputinar}
    \forall \N, k=1..N_c, \quad \M[h_k\bm] \succeq 0.
\end{equation}

For a finite $d \in \N$, a sufficient condition for \eqref{eq:truncputinar} to hold is that the upper-left corner of $\M[h_k \bm]$ containing moments up to degree $2d$ (expressed as $\M_d[h_k \bm]$) is \ac{PSD}. The truncated matrix  $\M_d[h_k\bm]$ represents the bilinear operator $L_{h \bm}$ in the finite-dimensional vector space $\R[x]_{\leq d}$, and has size $\binom{n+d}{d}$ when $\{e_i\}$ is a monomial basis.

We will define the degree-$d$ block-diagonal matrix composed of localizing matrices for constraints of $\K$ as
\begin{equation}
    \M_d[\K\bm] = \mathrm{diag}(\M_{d-\lceil d_k/2 \rceil} [h_k \bm])_{k=1..N_c}. \label{eq:block_synthetic}
\end{equation}

The moment-\ac{SOS} hierarchy is the process of raising the degree $d$ to $\infty$ while imposing that the matrix in \eqref{eq:block_synthetic} is \ac{PSD}.

\subsection{Chance-Peak Moment Setup}


We require polynomial-structuring assumptions in order to approximate \eqref{eq:var_meas_soc} and \eqref{eq:peak_cvar_meas} using the moment-\ac{SOS} hierarchy:

\begin{itemize}
    \item[A8] The sets $X_0$ and $X$ are both \ac{BSA} with ball constraints.
    \item[A9] The generator $\Lie$ satisfies $\forall v \in \R[t, x], \ \Lie v \in \R[t, x]$.
    \item[A10] The objective function $p(x)$ is also polynomial.
\end{itemize}

Let $(\bm, \bm^\tau)$ be pseudomoments for the optimization variables $(\mu, \mu_\tau)$. For each $\alpha \in \N^n$ and $\beta \in \N$ producing the monomial $x^\alpha t^\beta$, we define the operator $\mathcal{D}_{\alpha \beta}$ as the moment expression for the generator $\Lie$ with
\begin{equation}
    \label{eq:dyn_mom}
    \mathcal{D}_{\alpha \beta}(\bm, \bm^\tau) =
        \bm^\tau_{\alpha\beta} - L_\bm(\Lie(x^\alpha t^\beta)).
\end{equation}

We further define a dynamics degree $D$ as a function of the degree $d$ such that 
\begin{align}
    \deg x^\alpha t^\beta \leq 2d \implies \deg \Lie x^\alpha t^\beta \leq 2D.
\end{align}






\subsection{Tail-Bound Moment Program}

\begin{prob}
For any $d \in \N$ such that $2d \geq \deg p$, the order-$d$ \ac{LMI} in pseudomoments to upper-bound \eqref{eq:var_meas_soc} given $\mu_0$ is
\begin{subequations}
\label{eq:chance_lmi}
\begin{align}
    p^*_{r, d} = & \max \quad r {c} +  L_{\bm^\tau} p  \label{eq:chance_lmi_obj} \\
    & {c} \in \R, \bm \in \R^{^{\binom{2D+n+1}{n+1}}}, \bm^\tau \in \R^{^{\binom{2d+n+1}{n+1}}} \notag \\
    &\mathcal{D}_{\alpha \beta}(\bm, \bm^\tau) = \delta_{\beta 0} \langle x^\alpha , \mu_0\rangle \notag \\
    & \qquad \forall (\alpha, \beta) \in \N^{n+1} \quad \text{s.t.} \quad |\alpha|+\beta \leq 2d \label{eq:chance_lmi_flow}\\
    & {y}= [1-L_{\bm^\tau}(p^2), \ 2 {c}, \ 2 L_{\bm^\tau}p] \\
    & ({y}, 1 + L_{\bm^\tau}(p^2)) \in {\mathbb{L}}^3 \label{eq:var_lmi_con_soc}\\
    &\M_d[([0, T] \times X)\bm^\tau] \succeq 0 \label{eq:lmitau} \\
    &\M_D[([0, T] \times X)\bm] \succeq 0, \label{eq:lmiocc}
\end{align}
\end{subequations}
\end{prob}
in which the Kronecker symbol $\delta_{\beta 0}$ is $1$ if $\beta = 0$ and is $0$ otherwise.
Constraint \eqref{eq:chance_lmi_flow} is a finite-dimensional truncation of the infinite-dimensional martingale relation \eqref{eq:var_meas_lie_soc}.


We now prove that all optimization variables of \eqref{eq:var_meas_soc} are bounded in order to prove convergence of \eqref{eq:chance_lmi} as $d\rightarrow \infty$.

\begin{lem}
\label{lem:moment_bound}
The variables $(\mu, \mu_\tau, {c})$ in any feasible solution of \eqref{eq:var_meas_soc} are bounded under A1-A7.
\end{lem}
\begin{proof}
A measure is bounded if all of its moments are bounded (finite). A sufficient condition for boundedness of measures to hold is if the measure is supported on a compact set and that it has finite mass. Compact support of $(\mu, \mu_\tau)$ is ensured by A1. Substitution of $v(t, x) = 1$ into \eqref{eq:var_meas_lie_soc} leads to $\inp{1}{\mu_\tau}=\inp{1}{\mu_0} = 1$ by A4-5. The moments $\inp{p}{\mu_\tau}$ and $\inp{p^2}{\mu_\tau}$ are therefore finite under A1 and A3, which implies that $c$ is bounded as well.
Applying $v(t, x) = t$ yields $\inp{1}{\mu} \leq \inp{t}{\mu_\tau} \leq T$, which implies that $\inp{1}{\mu}$ is finite by A5. All variables are therefore bounded.

\end{proof}

\begin{thm}
\label{thm:lmi_converge}
Under assumptions A1-A10, \eqref{eq:chance_lmi} inherits the strong duality property of its infinite-dimensional counterpart \eqref{eq:var_meas_soc}, and its optima will converge to \eqref{eq:var_meas_soc} i.e. $\lim_{d \rightarrow \infty} p^*_{r, d} = p^*_r$.
\end{thm}
\begin{proof}
Strong duality is proved almost identically in the finite dimensional setting as in the infinite dimensional setting of Theorem \ref{thm:cdc_strong_dual} by using the same arguments as in the proof of \cite[Proposition 6]{tacchi2022convergence}.

Convergence is a direct consequence of \cite[Corollary 8]{tacchi2022convergence} (when extending to the case with finite-dimensional \ac{SOC} variables) through Lemma \ref{lem:moment_bound}.
\end{proof}

\begin{rmk}
The relation $p^*_d \geq p^*_r \geq P^*_r$ will still hold when $[0, T] \times X$ is noncompact (violating A1 and A5), but it may no longer occur that $\lim_{d \rightarrow \infty} p^*_{r, d} = p^*_r$ (the conditions Lemma \ref{thm:lmi_converge} will no longer apply).
\end{rmk}

\subsection{CVAR Moment Program}

Let $(\bm, \bm^\tau, {\mathbf{n}}, \hat{{\mathbf{n}}})$ be respective moment sequences of the measures $(\mu, \mu_\tau, {\nu, \hat{\nu}})$. 
We define the operator $\mathcal{E}_{k}(\bm^\tau, {\mathbf{n}}, \hat{{\mathbf{n}}})$ for ${k} \in \N$ as the moment counterpart of the operator in constraint \eqref{eq:peak_cvar_meas_cvar}:
\begin{align}
\label{eq:cvar_mom}
    \mathcal{E}_{k}(\bm^\tau, {\mathbf{n}}, \hat{{\mathbf{n}}}) = L_{\bm^\tau}(p(x)^{k}) - \epsilon {\mathbf{n}}_{k} - \hat{{\mathbf{n}}}_{k}.
\end{align}

The associated \ac{CVAR} degree ${\Delta}$ is
\begin{align}
    {\Delta} = \lfloor d/ \deg p \rfloor. 
\end{align}


\begin{prob}
For each degree $d \geq \deg p$, the order-$d$ moment \ac{LMI} that upper-bounds the \ac{CVAR} \ac{LP} \eqref{eq:peak_cvar_meas} given the distribution  $\mu_0$ is
\begin{subequations}
\label{eq:cvar_lmi}
\begin{align}
    p^*_{c, d} = & \max L_{{\mathbf{n}}} q  \label{eq:cvar_lmi_obj} \\
    & \bm \in \R^{^{\binom{2D+n+1}{n+1}}}, \bm^\tau \in \R^{^{\binom{2d+n+1}{n+1}}} \notag \\
    & {\mathbf{n}}\in \R^{2{\Delta}+1} , \  \hat{{\mathbf{n}}}\in \R^{2{\Delta}+1} \notag \\
    & \forall (\alpha, \beta) \in \N^{n+1} \quad \text{s.t.} \quad |\alpha|+\beta: \leq 2d  \nonumber \\
    &\qquad \mathcal{D}_{\alpha \beta}(\bm, \bm^\tau) = \delta_{\beta 0} \langle x^\alpha , \mu_0\rangle  \qquad \label{eq:cvar_lmi_flow}\\
    & \mathcal{E}_{k}(\bm^\tau, {\mathbf{n}}, \hat{{\mathbf{n}}}) = 0 \qquad \forall {k} \in 0..(2{\Delta}) \label{eq:cvar_lmi_cvar}\\
    &{\mathbf{n}}_0 = 1 \\
    &\M_d[([0, T] \times X)\bm^\tau]\succeq  0\label{eq:lmitauocc}\\
    &\M_D[([0, T] \times X)\bm] \succeq  0  \\
    & \M_{\Delta}[{[\hat{p},\check{p}]\mathbf{n}}] \succeq 0 \label{eq:lmicvar1}\\
    & \M_{\Delta}[{[\hat{p},\check{p}]\mathbf{n}}] \succeq 0 \label{eq:lmicvar2}.
\end{align}
\end{subequations}
\end{prob}
The symbol $\delta_{\beta 0}$ is the Kronecker Delta ($1$ if $\beta = 0$ and $0$ otherwise). Constraints \eqref{eq:cvar_lmi_flow}  and \eqref{eq:cvar_lmi_cvar} are finite-dimensional truncations of constraints \eqref{eq:peak_cvar_meas_flow} 
 and \eqref{eq:peak_cvar_meas_cvar}  respectively.

\begin{rmk}
    The redundant support constraints in \eqref{eq:lmicvar1},\eqref{eq:lmicvar2} are added to ensure that the Archimedean condition is satisfied for each moment sequence support.
\end{rmk}

\begin{thm}
\label{thm:bounded_mass_cvar}
    All measures $\mu, \mu_\tau, {\nu, \hat{\nu}}$ are bounded under A1-A4.
\end{thm}
\begin{proof}
    This boundedness will be proved by showing that the mass of each measure is bounded and their support sets are compact.

    \textit{Compactness:} The measures $\mu_\tau, \mu$ have compact support under A1. The quantities ${\hat{p}} = \min_{x \in X} p(x)$ and ${\check{p}} = \max_{x \in X} p(x)$ are each finite and attained under A1 and A3. The measure $p_\# \mu_\tau$ has support inside the compact set $[{\hat{p}}, {\check{p}}]$ given that $\supp{\mu_\tau} \subseteq [0, T] \times X$. Both ${\nu}$ and ${\hat{\nu}}$ are nonnegative measures with $\epsilon > 0$, so \eqref{eq:peak_cvar_meas_cvar} {ensures} that $\supp{\nu} \subseteq [{\hat{p}}, {\check{p}}]$ and $ \supp{\hat{\nu}} \subseteq [{\hat{p}}, {\check{p}}]$.

    \textit{Bounded Mass:} The mass of  ${\nu}$ is set to 1 by \eqref{eq:peak_cvar_meas_prob}. Passing a test function $v(t, x) = 1$ through \eqref{eq:peak_cvar_meas_flow} results in $\inp{1}{\mu_\tau} = \inp{1}{\mu_0}$, which equals 1 by A4. Substitution of $v(t, x) = t$ through the same constraint yields $0 + \inp{1}{\mu} = \inp{t}{\mu_\tau} \leq T$. Given that $\mu_\tau$ is a probability measure, it holds that $p_\# \mu_\tau$ is also a probability measure. Finally, we analyze the mass of {$\hat\nu$ with} \eqref{eq:peak_cvar_meas_cvar}: 
    {
    \begin{align*}
        \inp{1}{\hat{\nu}} &= \inp{1}{p_\#\mu_\tau - \epsilon \nu}
        \\
        &= \inp{1}{p_\#\mu_\tau} - \epsilon\inp{1}{\nu}\\
        &= 1-\epsilon < \infty.
    \end{align*}
    }
    All nonnegative measures have compact support and bounded masses under A1-A4, proving the theorem.
\end{proof}

\begin{thm}
\label{thm:cantelli_cvar_relation}
    The Cantelli $(r = \sqrt{1-1/\epsilon})$ program from \eqref{eq:var_meas_soc} with objective $p^*_r$ will upper-bound the \ac{CVAR} program \eqref{eq:peak_cvar_meas} by $p^*_r \geq p^*_c$.
\end{thm}
\begin{proof}
    This holds by Equation (5) of \cite{vcerbakova2006worst}. For a general univariate {random variable $V \sim \psi \in \Mp{\R}$} with variance $\sigma = \sqrt{\mathbb{E}[V^2]-\mathbb{E}[V]^2}$, the Cantelli bound is related to \ac{CVAR} by
    \begin{align} \label{eq:cantelli_cvar_relation}
        {\mathbb{E}[V]} + \sigma \sqrt{1/\epsilon - 1} \geq { \mathrm{ES}_\epsilon(V) \geq \mathrm{VaR}_\epsilon(V)}.
    \end{align}

    The Cantelli bound obtains the worst-case \ac{CVAR} among all possible distributions agreeing with the given first and second moments \cite{dupacova1987minimax}.
\end{proof}

 \begin{thm}
     The upper bounds of \eqref{eq:cvar_lmi} will satisfy $p^*_d \geq p^*_{d+1} \geq ... \geq p^*$ and $\lim_{d \rightarrow \infty} p^* = P^*$ under assumptions A1-A10.
 \end{thm}
 \begin{proof}
     Boundedness and convergence will occur through Corollary 8 of \cite{tacchi2022convergence} (all sets are Archimedean, all data is polynomial, measure solutions are bounded by Theorem \eqref{thm:bounded_mass_cvar}, and the true objective value is finite).
 \end{proof}

     

\subsection{Computational Complexity}
In Problem \eqref{eq:chance_lmi}, the computational complexity mostly depends on the number and size of the matrix blocks involved in LMI constraints (\ref{eq:lmitau},\ref{eq:lmiocc}), which in turn depend on the number and degrees of polynomial inequalities describing $X$ (the higher $d_k = \deg(h_k)$, the smaller $\M_{d-\lceil d_k / 2 \rceil}[h \bm]$). At order-$d$, the maximum size of localizing matrices is $\binom{n+1+ D}{D}$. This same analysis occurs for Problem \eqref{eq:cvar_lmi} with respect to the \ac{LMI} constraints \eqref{eq:lmitauocc}-\eqref{eq:lmicvar2}.


Problems \eqref{eq:chance_lmi}  and \eqref{eq:cvar_lmi} must be converted to \ac{SDP}-standard form by introducing equality constraints between the entries of the moment matrices in order to utilize symmetric-cone Interior Point Methods (e.g., Mosek \cite{mosek92}). The per-iteration complexity of \iac{SDP} involving a single moment matrix of size $\binom{n+d}{d}$ scales as $n^{6d}$ \cite{lasserre2006pricing}. The scaling of \iac{SDP} with multiple moment and localizing matrices generally depends on the maximal size of any \ac{PSD} matrix. 
In our case, this size is at most $\binom{n+1+d}{d}$ with a scaling impact of $(n+1)^{6 d}$ or $d^{4(n+1)}$. The complexity of using this chance-peak routine increases in a jointly polynomial manner with $d$ and  $n$.



\section{Extensions}
\label{sec:extensions}
This section outlines extensions to the developed chance-peak framework. The formulas in this section will focus on the \ac{CVAR} programs, but similar expressions may be derived for the tail-bound programs.

\subsection{Switching}


The \ac{CVAR}-peak scheme may also be applied to switched stochastic systems. The methods outlined in this section are an extension of the \ac{ODE} approach from \cite{miller2021uncertain}, and are similar to duals of constraints found in \cite{prajna2007framework}.
Assume that there are $N_s \in \N$ subsystems indexed by $\ell=1..N_s$, each with an individual generator $\Lie_\ell$. As an example, a switched system with \ac{SDE} subsystems $\ell = 1..N_s$ could have individual dynamics
\begin{align}
    dx &= f_\ell(t, x) dt + g_\ell(t, x) \label{eq:dyn_switch}
\end{align}
and have associated generators $\Lie_\ell$ 
mapping $\forall v \in C^{1,2}([0, T] \times X) = \cs([0, T] \times X), \forall \ell=1..N_s$:
\begin{align}
    \Lie_\ell v(t, x) &= \partial_t v + f_\ell\cdot \nabla_x v + g_\ell^T(\nabla_{xx}^2v) g_\ell/2.
\end{align}

A switched trajectory is a distribution $x(t)$ and a switching function $S: [0, T] \rightarrow (1..N_s)$ under the constraint that $x(t)$ satisfies \eqref{eq:dyn_switch} whenever $S(t) = \ell$ (the $\ell$-th subsystem is active). A specific trajectory of a switched process starting from an initial point $x_0 \in X$ will be expressed as $x(t \mid x_0, S)$. No dwell time constraints are imposed on the switching sequence $S$; instead, switching can occur arbitrarily quickly in time.


Let $\mu \in \Mp{[0, T]\times X}$ be the total occupation measure of the switched process trajectory $x(t \mid x_0, S)$.
The total occupation measure may be split into disjoint subsystem occupation measures $\forall \ell: \ \mu_\ell \in \Mp{[0, T] \times X}$ under the relation $\sum_{\ell=1}^{N_s} \mu_\ell = \mu$. 
The mass of a subsytem's occupation measure $\inp{1}{\mu_\ell}$ is the total amount of time that the trajectory $x(t \mid x_0, S)$ spends in subsystem $S(t) =\ell$.

The martingale equation (generalization of Dynkin's \eqref{eq:dynkin}) for switching-type uncertainty is
\begin{align}
    \mu_\tau = \delta_0 \otimes \mu_0 + \textstyle \sum_{\ell=1}^{N_s} \Lie_\ell^\dagger \mu_\ell.
\end{align}

The \ac{CVAR}-peak problem in \eqref{eq:peak_cvar_meas} modified  for switching uncertainty is
\begin{subequations}
\label{eq:peak_cvar_switch_meas}
{
\begin{align}
p^* = & \ \sup \quad \inp{\mathrm{id}_\R}{\nu} \label{eq:peak_cvar_switch_meas_obj} \\
    & \mu_\tau = \delta_0 \otimes\mu_0 + \textstyle \sum_{\ell=1}^L \Lie_\ell^\dagger \mu_\ell \label{eq:peak_cvar_switch_meas_flow}\\
    & \inp{1}{\nu} = 1\label{eq:peak_cvar_switch_meas_prob}\\
    & \epsilon \nu + \hat{\nu} = p_\# \mu_\tau \label{eq:peak_cvar_switch_meas_cvar}\\
    & \forall \ell\in 1..L: \mu_\ell \in \Mp{[0, T] \times X}  \\
    & \mu_\tau \in \Mp{[0, T] \times X} \label{eq:peak_cvar_switch_meas_peak}\\    
    & \nu, \hat{\nu} \in \Mp{\R}. \label{eq:peak_cvar_switch_meas_slack}
\end{align}
}
\end{subequations}

\subsection{Distance Estimation}
\label{sec:distance}

The \ac{CVAR}-peak methodology developed in this paper can be applied towards bounding (probabilistically) the distance of closest approach to an unsafe set. Let $X_u \subset X$ be an unsafe set, and let ${(x,x') \mapsto\,} c(x, {x'})$ be a metric in $X$. The point-set distance function with respect to $X_u$ is $c(x; X_u) = \inf_{{x'} \in X_u} c(x, {x'})$. 


The $\epsilon$-\ac{CVAR} distance program may be expressed as 
\eqref{eq:peak_cvar_meas} with an infimal (rather than supremal) objective $p(x) =c(x; X_u)$. 
Because the objective $c({\bullet}; X_u)$ is not generally polynomial (even when $c$ is polynomial), the \ac{LMI} \eqref{eq:chance_lmi} cannot directly be posed in terms of $c({\bullet}; X_u)$. One method to maintain a polynomial structure is to add time-constant states $dx_u = \0 dt + \0 dw$ to dynamics \eqref{eq:sde} in $x$ and form the state support set $(x, {x_u}) \in X \times X_u$. When $X_u$ is full-dimensional inside $X \subset \R^n$, the occupation measure $\mu \in \Mp{[0, T] \times X \times {X_u}}$ will have a moment matrix of size $\binom{1+2n+d}{d}$ at each fixed degree $d$.

This size can be reduced using the method in \cite{miller2022distance_short}, in which the peak measure $\hat{\mu}_\tau \in \Mp{[0, T] \times X \times {X_u}}$ is decomposed into a joint measure $\eta \in \Mp{X \times {X_u}}$  and a peak measure $\mu_\tau \in \Mp{[0, T] \times X}$ that have equal $x$ marginals. 
The resultant \ac{CVAR}-distance \ac{LP} is
\begin{subequations}
\label{eq:peak_cvar_dist_meas}
{
\begin{align}
c^* = & \ \inf \quad \inp{\mathrm{id}_\R}{\nu} \label{eq:peak_cvar_dist_meas_obj} \\
    & \mu_\tau = \delta_0 \otimes\mu_0 + \Lie^\dagger \mu \label{eq:peak_cvar_dist_meas_flow}\\
    &\forall \phi \in C(X): \label{eq:peak_cvar_dist_meas_marginal}\\
    & \quad \int_{[0,T]\times X} \phi(x) \; d\mu_\tau(t, x) = \int_{X\times Y} \phi(x) \; d\eta(x, y)  \nonumber & & \\
    & \inp{1}{\nu} = 1\label{eq:peak_cvar_dist_meas_prob}\\
    & \epsilon \nu + \hat\nu = c_\#\mu_\tau \label{eq:peak_cvar_dist_meas_cvar}\\
    & \mu, \mu_\tau \in \Mp{[0, T] \times X} \notag\\ 
    & \nu, \hat{\nu} \in \Mp{\R} \notag \\
    & \eta \in \Mp{X \times X_u}. \notag
\end{align}
}
\end{subequations}

Constraint \eqref{eq:peak_cvar_dist_meas_marginal} enforces equality in the $x$-marginals between $\mu_\tau$ and $\eta$. Constraint \eqref{eq:peak_cvar_dist_meas_cvar} is a distance analogue of the pushforward \ac{CVAR} constraint in \eqref{eq:peak_cvar_meas_cvar}.
The Moment matrices of $\eta$ and $\mu$ respectively in the \ac{LMI} program derived from \eqref{eq:peak_cvar_dist_meas} have sizes $\binom{2n+d}{d}$,  $\binom{n+1+\tilde{d}}{\tilde{d}}$, and $({\Delta}+1)$. Unfortunately, the exponentiation operation $\inp{c^{k}}{\eta}$ causes mixed multiplications in variables even when $c$ is additively separable as $c(x, y) = \sum_{i=1}^n c_i(x_i, y_i)$ (Section V of \cite{miller2021distance}), thus forbidding the application of correlative sparsity \cite{waki2006sums} to reduce the complexity of \acp{LMI} from \eqref{eq:peak_cvar_dist_meas}.

\section{Numerical Examples}

\label{sec:examples}

All experiments were written in MATLAB (2022a)  and require Mosek \cite{mosek92} and YALMIP \cite{lofberg2004yalmip} dependencies. \ac{MC} sampling  is conducted with 50,000 sample paths (with \ac{SDE} parameters of antithetic sampling and a spacing of $\Delta t = 10^{-3}$) to approximate \ac{VAR} and \ac{CVAR} estimates. All experiments are accompanied by tables of chance-peak bounds and solver times.
Files to generate examples are available at \url{https://github.com/Jarmill/chance\_peak} (tail-bound Cantelli/\ac{VP}) and \url{https://github.com/Jarmill/cvar\_peak} (\ac{CVAR}). 

\subsection{Two States}
The first experiment is a cubic polynomial \ac{SDE} from Example 1 of \cite{prajna2004stochastic}: 
\begin{equation}
\label{eq:flow_sde}
    dx = \begin{bmatrix}x_2 \\ -x_1 -x_2 - \frac{1}{2}x^3_1\end{bmatrix}dt + \begin{bmatrix} 0 \\ 0.1 \end{bmatrix}dW.
\end{equation}

Chance-peak maximization of $p(x) = -x_2$ will occur starting at the point (Dirac-delta initial measure $\mu_0$) $X_0 = [1, 1]$ in a state set of $X= [-1, 2] \times  [-1, 1.5]$ and time horizon of $T=5$. Figure \ref{fig:sde_flow} displays trajectories of \eqref{eq:flow_sde} in cyan, starting from the black-circle point $X_0$. Four of these trajectory sample paths are marked in non-cyan colors.

The `mean' row of Tables \ref{tab:flow_sde} and \ref{tab:flow_sde_cvar} display upper-bounds on the mean as solved by \ac{SDP} truncations of \eqref{eq:peak_meas}. 
The bounds at $\epsilon = \{0.15, 0.1, 0.05\}$ for Table \ref{tab:flow_sde} are acquired by using the \ac{VP} expression in \eqref{eq:var_vp} and solving the \acp{SDP} obtained from \eqref{eq:chance_lmi}.  Table \ref{tab:flow_sde_cvar} displays bounds from \acp{SDP} derived from the \ac{CVAR} \ac{LMI} \eqref{eq:cvar_lmi}.
The dash-dot red, dotted black, and solid red lines in Figure \ref{fig:sde_flow} are the mean, $\epsilon=0.15$ \ac{CVAR}, and $\epsilon=0.15$ \ac{VP} bounds respectively at order 6. All subsequent plots will retain this coloring and styling scheme for the mean, \ac{CVAR}, and \ac{VP} bounds. The top-right entry of Table \ref{tab:flow_sde_cvar} (and all similar tables) are omitted to reduce confusion between the $\epsilon=0.5$ VAR bound (which equals the mean) and the $\epsilon=0.5$ \ac{CVAR} bound (which can exceed the mean).


\begin{table}[h]
   \centering
   \caption{\ac{VP} Chance-Peak estimation of the Stochastic Flow System \eqref{eq:flow_sde} to maximize $p(x) = -x_2$ \label{tab:flow_sde}}
\begin{tabular}{lcccccc}
\multicolumn{1}{c}{order}        & 2      & 3      & 4      & 5      & 6   & \ac{VAR} \ac{MC}   \\ \hline
$\epsilon = 0.5$ & 0.8818 & 0.8773 & 0.8747 & 0.8745 & 0.8744 & 0.8559\\
$\epsilon = 0.15$ & 1.6660 & 1.6113 & 1.5842 & 1.5771 & 1.5740 & 0.9142\\
$\epsilon = 0.1$  & 2.0757 & 1.9909 & 1.9549 & 1.9461 & 1.9427 & 0.9279\\
$\epsilon = 0.05$  & 2.9960 & 2.8441 & 2.7904 & 2.7772 & 2.7715 &  0.9484 \\
\end{tabular}
\end{table}

\begin{figure}
    \centering
    \includegraphics[width=0.9\exfiglength]{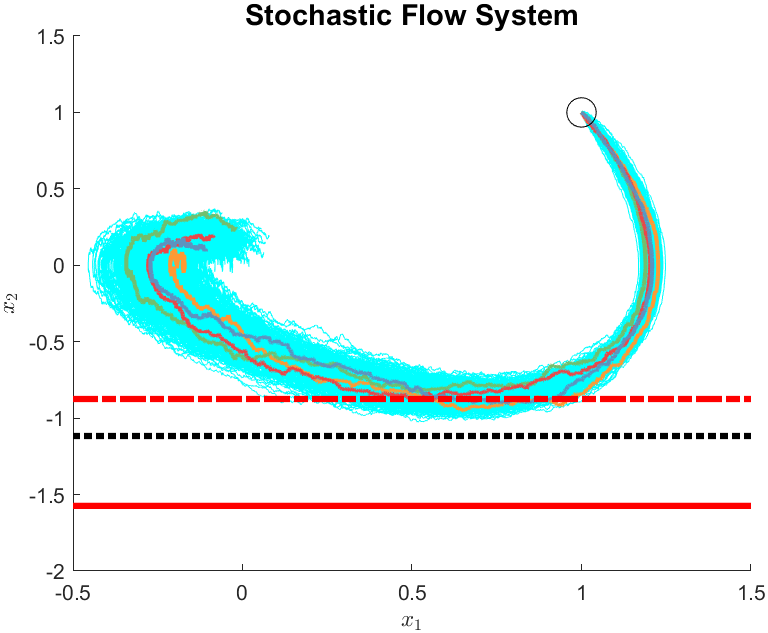}
    \caption{Trajectories of \eqref{eq:flow_sde} with mean (dash-dot red), \ac{CVAR} $\epsilon=0.15$  (dotted black), and \ac{VP} $\epsilon=0.15$ (solid red) bounds}
    \label{fig:sde_flow}
\end{figure}

\begin{table}[!h]
   \centering
   \caption{Solver time (seconds) to compute Table \ref{tab:flow_sde} \label{tab:flow_sde_time}}
\begin{tabular}{lccccc}
\multicolumn{1}{c}{order}      & 2      & 3      & 4      & 5      & 6      \\ \hline
$\epsilon = 0.5$   & 0.380 & 0.449 & 0.625 & 1.583 & 4.552 \\
$\epsilon = 0.15$  & 0.262 & 0.443 & 0.727 & 2.756 & 5.586 \\
$\epsilon = 0.1$   & 0.268 & 0.380 & 1.364 & 2.882 & 3.143 \\
$\epsilon = 0.05$ & 0.242 & 0.390 & 1.261 & 2.923 & 7.539
\end{tabular}
\end{table}

\begin{table}[!h]
   \centering
   \caption{\ac{CVAR} Chance-Peak estimation of the Stochastic Flow System \eqref{eq:flow_sde} to maximize $p(x) = -x_2$ \label{tab:flow_sde_cvar}}
\begin{tabular}{lcccccc}
\multicolumn{1}{c}{order}      & 2      & 3      & 4      & 5      & 6  & \ac{CVAR} \ac{MC}     \\ \hline
mean                      & 0.8818                & 0.8773                & 0.8747                & 0.8745                & 0.8744                &                                            \\
$\epsilon = 0.15$         & 1.2500                & 1.2500                & 1.1655                & 1.1313                & 1.1170                & 0.9432
                                 \\
$\epsilon = 0.1$          & 1.2500                & 1.2500                & 1.2116                & 1.1666                & 1.1466                & 
    0.9546
                                                \\
$\epsilon = 0.05$         & 1.2500                & 1.2500                & 1.2500                & 1.2266                & 1.1959                &    0.9720                                           

\end{tabular}
\end{table}

\begin{table}[!h]
   \centering
   \caption{Solver time (seconds) to compute Table \ref{tab:flow_sde_cvar} \label{tab:flow_sde_cvar_time}}
\begin{tabular}{lccccc}
\multicolumn{1}{c}{order}      & 2      & 3      & 4      & 5      & 6      \\ \hline
mean                      & 0.619                 & 0.504                 & 0.558                 & 2.862                 & 1.652                 \\
$\epsilon = 0.15$         & 0.363                 & 0.402                 & 0.399                 & 1.341                 & 2.016                 \\
$\epsilon = 0.1$          & 0.357                 & 0.442                 & 0.417                 & 1.054                 & 2.093                 \\
$\epsilon = 0.05$         & 0.414                 & 0.393                 & 0.399                 & 0.579                 & 2.534                

\end{tabular}
\end{table}

\subsection{Three States}

A stochastic version of the Twist dynamics in \cite{miller2022distance_short} is:

\begin{equation}
\label{eq:twist_sde}
    dx = \begin{bmatrix}-2.5x_1 + x_2 - 0.5x_3 + 2x_1^3+2x_3^3 \\
    -x_1+1.5x_2+0.5x_3-2x_2^3-2x_3^3 \\
    1.5 x_1 + 2.5x_2 - 2 x_3 - 2x_1^3 - 2 x_2^3\end{bmatrix}dt + \begin{bmatrix} 0 \\ 0\\  0.1 \end{bmatrix}dW.
\end{equation}


This example applies the chance-peak setting towards maximization of $p(x) = x_3$, with an initial condition of $X_0 = [0.5, 0, 0]$ and a state set $X=[-0.6, 0.6] \times  [-1, 1] \times [-1, 1.5]$ and $T=5$. 
\ac{VP} and \ac{CVAR} bounds from the \eqref{eq:peak_meas} and \eqref{eq:chance_lmi} \acp{SDP} are written in Tables \ref{tab:twist_sde} and \ref{tab:twist_sde_cvar}, similar in format to Tables \ref{tab:flow_sde} and 
\ref{tab:flow_sde_cvar}. Trajectories and  bounds are plotted in \ref{fig:sde_twist} beginning from the black-circle $X_0$. The three planes are order-6 bounds at $\epsilon = 0.15$, in which the top solid red plane is the \ac{VP} bound, the translucent black plane is the $\epsilon=0.15$ \ac{CVAR} bound at order 6,  and the translucent red plane is the mean bound on $x_3$.


\begin{table}[!h]
   \centering
   \caption{\ac{VP} Chance-Peak estimation of the Stochastic Twist System \eqref{eq:twist_sde}  to maximize $p(x) = x_3$ \label{tab:twist_sde}}
\begin{tabular}{lcccccc}
\multicolumn{1}{c}{order}            & 2      & 3      & 4      & 5      & 6     & \ac{VAR} \ac{MC} \\ \hline
$\epsilon = 0.5$   & 0.9100 & 0.8312 & 0.8231 & 0.8211 & 0.8201 & 0.7206\\
$\epsilon = 0.15$  & 1.6097 & 1.4333 & 1.3545 & 1.3318 & 1.3202 & 0.7685 \\
$\epsilon = 0.1$  & 1.9707 & 1.7453 & 1.6283 & 1.5877 & 1.5739 & 0.7801\\
$\epsilon = 0.05$  & 2.7834 & 2.4426 & 2.2333 & 2.1622 & 2.1267 & 0.7970
\end{tabular}
\end{table}


\begin{figure}[!h]
    \centering
    \includegraphics[width=0.9\exfiglength]{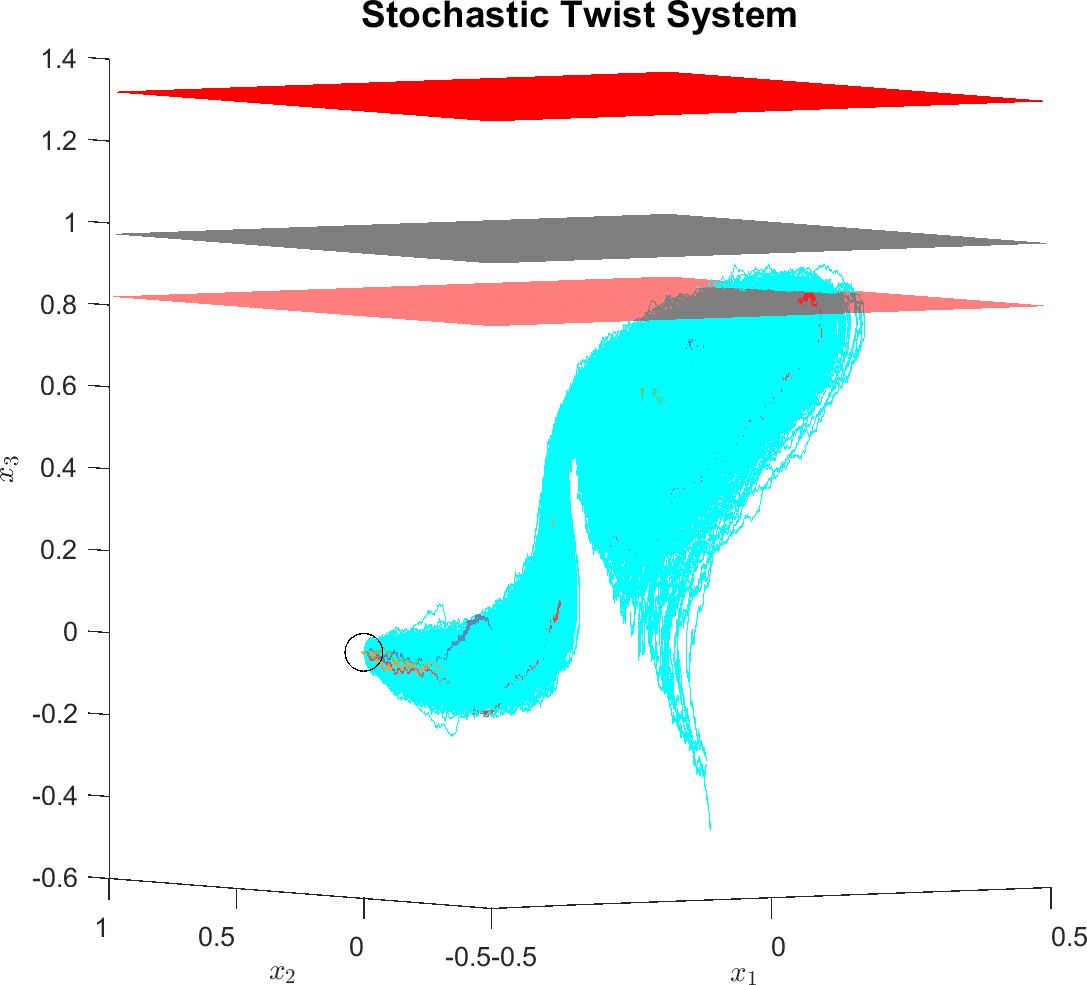}
    \caption{Trajectories of \eqref{eq:twist_sde} with mean \ac{VP} (solid red), mean \ac{CVAR} (translucent black), and $\epsilon=0.15$ (translucent red) bounds}
    \label{fig:sde_twist}
\end{figure}


\begin{table}[!h]
   \centering
   \caption{Solver time (seconds) to compute Table \ref{tab:twist_sde} \label{tab:twist_sde_time}}
\begin{tabular}{lccccc}
\multicolumn{1}{c}{order}         & 2      & 3      & 4      & 5      & 6      \\ \hline
$\epsilon = 0.5$   & 0.428 & 1.939 & 5.196 & 19.201 & 83.679  \\
$\epsilon = 0.15$  & 0.328 & 0.999 & 4.755 & 21.108 & 96.985  \\
$\epsilon = 0.1$  & 0.325 & 1.083 & 5.172 & 22.596 & 119.823 \\
$\epsilon = 0.05$  & 0.325 & 1.294 & 4.516 & 22.357 & 115.820
\end{tabular}
\end{table}

\begin{table}[!h]
   \centering
   \caption{\ac{CVAR} Chance-Peak estimation of the Stochastic Twist System \eqref{eq:twist_sde}  to maximize $p(x) = x_3$ \label{tab:twist_sde_cvar}}
\begin{tabular}{lcccccc}
\multicolumn{1}{c}{order}            & 2      & 3      & 4      & 5      & 6     & \ac{CVAR} \ac{MC} \\ \hline
mean                      & 0.9100                & 0.8312                & 0.8231                & 0.8211                & 0.8203                &                             \\
$\epsilon = 0.15$         & 1.4519                & 1.1251                & 1.0246                & 0.9892                & 0.9733                &     0.7923
           \\
$\epsilon = 0.1$          & 1.5850                & 1.1880                & 1.0613                & 1.0173                & 0.9950                &    
    0.8016
             \\
$\epsilon = 0.05$         & 1.8479                & 1.3063                & 1.1286                & 1.0646                & 1.0329                &     0.8156

\end{tabular}
\end{table}

\begin{table}[!h]
   \centering
   \caption{Solver time (seconds) to compute Table \ref{tab:twist_sde_cvar} \label{tab:twist_sde_time_cvar}}
\begin{tabular}{lccccc}
\multicolumn{1}{c}{order}         & 2      & 3      & 4      & 5      & 6      \\ \hline
mean                      & 0.761                 & 0.502                 & 1.845                 & 5.078                 & 27.543                \\
$\epsilon = 0.15$         & 0.386                 & 0.469                 & 0.996                 & 4.383                 & 35.634                \\
$\epsilon = 0.1$          & 0.330                 & 0.381                 & 1.030                 & 4.115                 & 20.513                \\
$\epsilon = 0.05$         & 0.328                 & 0.387                 & 1.014                 & 5.451                 & 26.280               

\end{tabular}
\end{table}

\subsection{Discrete-time System}

The prior examples involved \ac{SDE} dynamics. In this subsection, we focus on a discrete-time system in which the parameter $\lambda \in \R$ is sampled according to the unit normal distribution at each time step ($\lambda[t]\sim \mathcal{N}(0, 1)$). The discrete-time system system considered is
\begin{align}
    x_+ = \begin{bmatrix}
        -0.3 x_1  + 0.8x_2 + x_1 x_2 \lambda/4\\ -0.9x_1 - -0.1x_2 - 0.2 x_1^2 +  \lambda/40
    \end{bmatrix}.\label{eq:scatter_discrete}
\end{align}

This system involves a time horizon of $T=1$ with a time step of $\Delta t = 1/10$ (10 iterations after the initial condition). The initial point is $x_0 = [-1; 0.5]$, and trajectories evolve in a support set of $X = [-1.5, 1.5]^2$. Trajectories and order-6 bounds to maximize $p(x) = -x_2$ are displayed in Figure \ref{fig:scatter_discrete}.

\begin{figure}[h]
    \centering
    \includegraphics[width=\exfiglength]{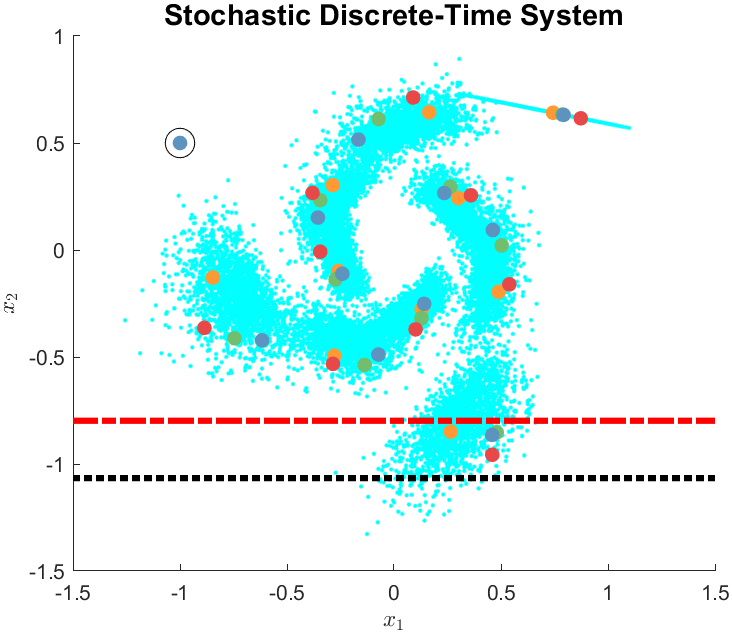}
    \caption{Trajectories of \eqref{eq:scatter_discrete} with $\epsilon = \{0.5, \textrm{ES} \ 0.15\}$ bounds }
    \label{fig:scatter_discrete}
\end{figure}

Table \ref{tab:twist_scatter} reports \ac{CVAR} bounds for $-x_2$ at orders $1$ to $6$. The \ac{VP} and Cantelli objectives produce bounds that are greater than $1.5$ (outside $X$) for the $d$ and $\epsilon$ values tested. As an example, the VP bound at $d=6, \epsilon = 0.15$ is $p(x) \leq -2.1252$

\begin{table}[!h]
   \centering
   \caption{\ac{CVAR} Chance-Peak estimation of the Discrete-time system \eqref{eq:scatter_discrete}  to maximize $p(x) = -x_2$ \label{tab:twist_scatter}}
\begin{tabular}{lcccccc}
\multicolumn{1}{c}{order}          & 2      & 3      & 4      & 5      & 6  & \ac{CVAR} MC \\ \hline
mean                                     & 0.8766                & 0.8128                & 0.8002                & 0.7982                & 0.7976     &          \\
$\epsilon = 0.15$                       & 1.5000                & 1.2139                & 1.0971                & 1.0743                & 1.0663  & 1.0287              \\
$\epsilon = 0.1$                          & 1.5000                & 1.2973                & 1.1446                & 1.1083                & 1.0997              & 1.0601  \\
$\epsilon = 0.05$                        & 1.5000                & 1.4500                & 1.2285                & 1.1653                & 1.1540  & 1.1092             

\end{tabular}
\end{table}

\subsection{Switching}

We utilize a modification of Example C from \cite{prajna2007framework} for this final example. The two subsystems involved are:
\begin{subequations}
\label{eq:switched_sde}
\begin{align}
    dx &= \begin{bmatrix}-2.5 x_1 - 2 x_2 \\ -0.5 x_1 - x_2 \end{bmatrix}dt + \begin{bmatrix}0 \\ 0.25 x_2 \end{bmatrix}dW \\
    dx &= \begin{bmatrix}-x_1 - 2 x_2 \\ 2.5 x_1 - x_2 \end{bmatrix}dt + \begin{bmatrix}0 \\ 0.25 x_2 \end{bmatrix}dW.
\end{align}
\end{subequations}

Switched \ac{SDE} trajectories start from an initial condition of $X_0 = (0, 1)$ and are tracked in the state set $X = [-2, 2]^2$  with a time horizon of $T=5$. The chance-peak problem is solved to find bounds on $p(x) = -x_2$.

Figure \eqref{fig:switched_sde} plots switched \ac{SDE} trajectories along with $\epsilon=\{0.5, \text{\ac{CVAR}} \ 0.15,  \text{\ac{VP}} \ 0.15\}$ bounds (at order-6). Tables \ref{tab:switched_sde} and \ref{tab:switched_sde_cvar} list these discovered bounds.

\begin{table}[!h]
   \centering
   \caption{\ac{VP} Chance-Peak upper-bounds for $p(x)=-x_2$ for the Switched System \eqref{eq:switched_sde} \label{tab:switched_sde}}
\begin{tabular}{lcccccc}
\multicolumn{1}{c}{order}                 & 2      & 3      & 4      & 5     & 6   & \ac{VAR} MC   \\ \hline
mean & 0.4304 & 0.3823 & 0.3630 & 0.3487 & 0.3352 & 0.0788\\
$\epsilon = 0.15$ & 0.9953 & 0.9328 & 0.9076 & 0.8918 & 0.8853 & 0.2384\\
$\epsilon = 0.1$   & 1.2888 & 1.2162 & 1.1865 & 1.1687 & 1.1609 & 0.2755\\
$\epsilon = 0.05$  & 1.9469 & 1.8516 & 1.8120 & 1.7891 & 1.7799 & 0.3334           
\end{tabular}
\end{table}

\begin{figure}[!h]
    \centering
    \includegraphics[width=\exfiglength]{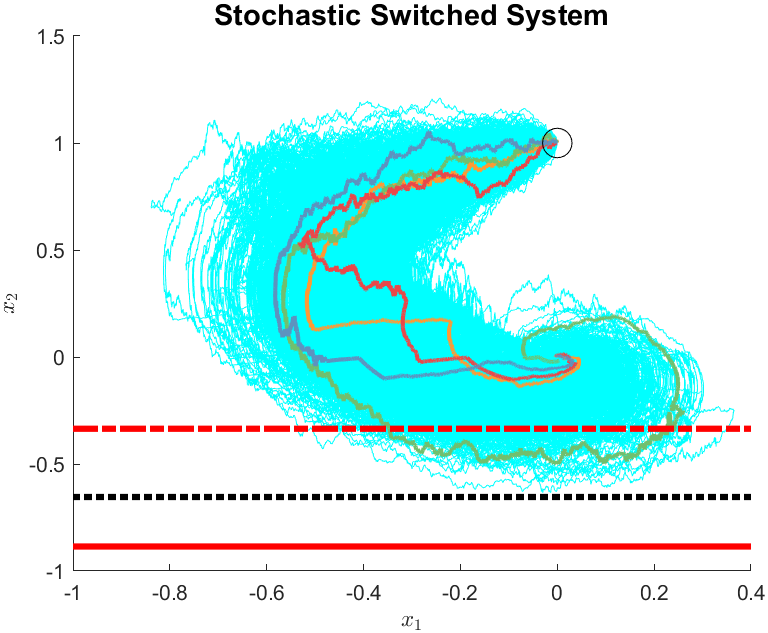}
    \caption{Trajectories of the switched system \eqref{eq:switched_sde} with $\epsilon=\{0.5, \text{\ac{CVAR}} \ 0.15,  \text{\ac{VP}} \ 0.15\}$ bounds}
    \label{fig:switched_sde}
\end{figure}

\begin{table}[!h]
   \centering
   \caption{Solver time (seconds) to compute Table \ref{tab:switched_sde} \label{tab:switched_sde_time}}
\begin{tabular}{lcccccc}
\multicolumn{1}{c}{order}          & 2      & 3      & 4      & 5      & 6      \\ \hline
mean & 0.362 & 0.389 & 0.570 & 1.755 & 2.499 \\
$\epsilon = 0.15$ &  0.257 & 0.295 & 0.587 & 1.812 & 3.718 \\
$\epsilon = 0.1$   & 0.237 & 0.281 & 1.636 & 2.364 & 3.191 \\
$\epsilon = 0.05$  & 0.251 & 0.291 & 0.906 & 1.735 & 2.638
\end{tabular}
\end{table}

\begin{table}[!h]
   \centering
   \caption{\ac{CVAR} Chance-Peak upper-bounds for $p(x)=-x_2$ for the Switched System \eqref{eq:switched_sde} \label{tab:switched_sde_cvar}}
\begin{tabular}{lcccccc}
\multicolumn{1}{c}{order}                 & 2      & 3      & 4      & 5     & 6   & \ac{CVAR} MC   \\ \hline
mean                                     & 0.4304                & 0.3823                & 0.3630                & 0.3488                & 0.3350                                       \\
$\epsilon = 0.15$         & 0.9698                & 0.7882                & 0.7157                & 0.6803                & 0.6540     & 0.4220                                  \\
$\epsilon = 0.1$                       & 1.1137                & 0.8818                & 0.7886                & 0.7433                & 0.7133   & 0.4470                                    \\
$\epsilon = 0.05$                         & 1.3924                & 1.0548                & 0.9207                & 0.8585                & 0.8200  & 0.4688                                    
\end{tabular}
\end{table}

\begin{table}[!h]
   \centering
   \caption{Solver time (seconds) to compute Table \ref{tab:switched_sde} \label{tab:switched_sde_cvar_time}}
\begin{tabular}{lcccccc}
\multicolumn{1}{c}{order}            & 2      & 3      & 4      & 5      & 6      \\ \hline
mean                                       & 0.471                 & 0.466                 & 0.674                 & 1.916                 & 3.310 \\
$\epsilon = 0.15$                          & 0.336                 & 0.371                 & 0.835                 & 1.954                 & 3.170 \\
$\epsilon = 0.1$                          & 0.330                 & 0.360                 & 0.816                 & 2.177                 & 4.237 \\
$\epsilon = 0.05$                         & 0.331                 & 0.373                 & 0.728                 & 1.486                 & 3.602

\end{tabular}
\end{table}

\subsection{Distance Estimation}


This example will involve distance estimation of a modification of the second subsystem of \eqref{eq:switched_sde}:
\begin{equation}
\label{eq:flowmod_sde}
    dx = \begin{bmatrix}-x_1 - 2 x_2 \\ 2.5 x_1 - x_2 \end{bmatrix}dt + \begin{bmatrix}0 \\ 0.1 \end{bmatrix}dW.
\end{equation}
This $L_2$ chance-distance  task  takes place at a time horizon of $T=5$ with sets $X_0 = [0; 0.75]$, $X = [-1.25, 1] \times [-1, 1]$, and $X_u = \{x_u \in \R^2 \mid 0.1^2 \geq (x_{u1} + 1)^2 + (x_{u2}+1)^2, \ x_{u1} + x_{u2} \leq -2\}$. Distance estimation was accomplished by maximizing \acp{VAR} of the function $-\norm{x-x_u}^2_2$ in \eqref{eq:var_meas_soc}, or by minimizing the \ac{CVAR} distance program in \eqref{eq:peak_cvar_dist_meas}.

System trajectories of \eqref{eq:flowmod_sde} are displayed in Figure \ref{fig:distance}, in which the unsafe half-circle set $X_u$ is drawn in solid red. Squared distance lower bounds from solving \acp{SDP} arising from distance moment programs  are listed in Table \ref{tab:flowmod_distance} and \ref{tab:flowmod_distance_cvar}. Negative distance lower-bounds are truncated to $0$ in Table \ref{tab:flowmod_distance}. This example demonstrates how the \ac{VP} chance-peak distance bounds for distance estimation are very conservative, and how \ac{CVAR} can offer an improvement in stochastic safety analysis.

\begin{table}[!h]
   \centering
   \caption{ \ac{VP} Chance-Peak squared distance lower bounds for System \eqref{eq:flowmod_sde} \label{tab:flowmod_distance}}
\begin{tabular}{lcccccc}
\multicolumn{1}{c}{order}                & 2      & 3      & 4      & 5       & 6     & VAR MC \\ \hline
mean &  1.1929 & 1.2337 & 1.2425 & 1.2490 & 1.2506 & 1.3162\\
$\epsilon = 0.15$  & 0 & 0 & 0 & 0.0182 & 0.0235 & 1.2432\\
$\epsilon = 0.1$  & 0 & 0 & 0 & 0 & 0 & 1.2261\\
$\epsilon = 0.05$  & 0 & 0 & 0 & 0 & 0  & 1.2012                   
\end{tabular}
\end{table}


\begin{figure}[!h]
    \centering    \includegraphics[width=0.9\exfiglength]{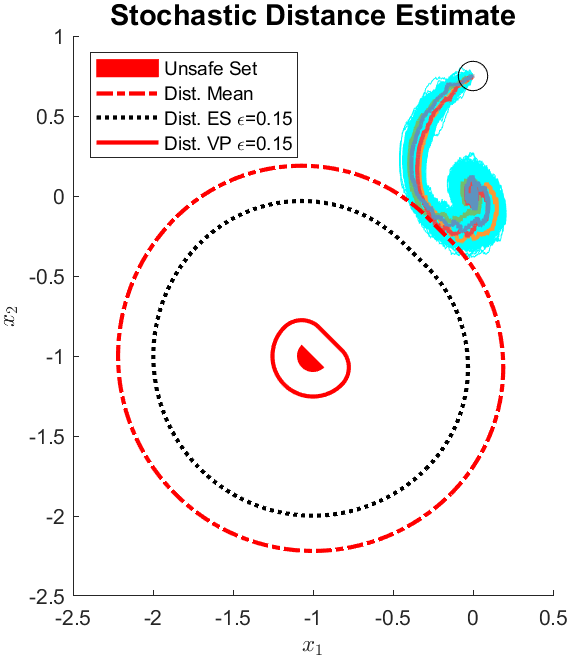}
    \caption{Trajectories of \eqref{eq:flowmod_sde} with $\epsilon = \{0.5, 0.15\}$ bounds}
    \label{fig:distance}
\end{figure}

\begin{table}[!h]
   \centering
   \caption{Solver time (seconds) to compute Table \ref{tab:flowmod_distance} \label{tab:flowmod_distance_time}}
\begin{tabular}{lccccc}
\multicolumn{1}{c}{order}            & 2      & 3      & 4      & 5      & 6      \\ \hline
mean &  0.507 & 0.512 & 1.772 & 6.569 & 21.331 \\
$\epsilon = 0.15$ &  0.346 & 0.453 & 1.233 & 5.836 & 23.930 \\
$\epsilon = 0.1$  &  0.344 & 0.482 & 1.522 & 5.172 & 21.034 \\
$\epsilon = 0.05$ &  0.384 & 0.485 & 1.711 & 5.954 & 26.974
\end{tabular}
\end{table}

\begin{table}[!h]
   \centering
   \caption{\ac{CVAR} Chance-Peak squared distance lower bounds for System \eqref{eq:flowmod_sde} \label{tab:flowmod_distance_cvar}}
\begin{tabular}{lcccccc}
\multicolumn{1}{c}{order}                & 2      & 3      & 4      & 5       & 6   & \ac{CVAR} MC \\ \hline
mean                          & 1.1929                & 1.2337                & 1.2427                & 1.2498                & 1.2494      &         \\
$\epsilon = 0.15$                         & 0                & 0                & 0.4884                & 0.5061                & 0.7980           & 1.2079     \\
$\epsilon = 0.1$           & 0                & 0                & 0.2825                & 0.3018                & 0.7047              & 1.1942  \\
$\epsilon = 0.05$                           & 0                & 0                                & 0                & 0                & 0.4952      & 1.1734         
                   
\end{tabular}
\end{table}

\begin{table}[!h]
   \centering
   \caption{Solver time (seconds) to compute Table \ref{tab:flowmod_distance_cvar} \label{tab:flowmod_distance_time_cvar}}
\begin{tabular}{lcccccc}
\multicolumn{1}{c}{order}      &      & 2      & 3      & 4      & 5      & 6      \\ \hline
mean                          & 0.272                 & 0.295                 & 1.229                 & 10.869                & 6.556                 \\
$\epsilon = 0.15$                          & 0.147                 & 0.207                 & 1.616                 & 4.403                 & 8.655                 \\
$\epsilon = 0.1$                           & 0.129                 & 0.187                 & 1.178                 & 4.987                 & 8.386                 \\
$\epsilon = 0.05$                         & 0.119                 & 0.193                 & 0.430                 & 1.382                 & 7.617                   

\end{tabular}
\end{table}

\section{Conclusion}

\label{sec:conclusion}

This paper considered the problem of finding the maximizing $\ac{VAR}$ of a state function $p(x)$ along trajectories of \ac{SDE} systems. Two upper-bounding methods were used to approximate this maximum \ac{VAR}: tail-bounds and \ac{CVAR}. The tail-bounds are used to formulate an \ac{SOCP} in measures \eqref{eq:var_meas_soc}, while \ac{CVAR} produces an \ac{LP} in measures \eqref{eq:peak_cvar_meas}. Each of these convex programs in measures are approximated by the moment-\ac{SOS} hierarchy of \acp{SDP} in a convergent manner if the problem data (dynamics, support sets) is polynomial.

Future work avenue involves reducing the conservatism of chance-peak based distance estimation, developing stochastic optimal control strategies to minimize quantile statistics, and applying chance-peak techniques towards analysis of hybrid systems.

\bibliographystyle{IEEEtran}
\bibliography{references.bib}

\appendix
\section{Strong Duality of Measure SDPs}
\label{app:duality_general}

 The work in \cite{tacchi2022convergence} gives sufficient condition to ensure strong duality in the framework of linear programming on measures.
This appendix generalizes the results of  \cite{tacchi2022convergence} by forming a framework of convex programming on measures with infinite-dimensional linear constraints and finite dimensional \ac{LMI} constraints on moments. In particular, we add the case of optimization over Borel measures with \ac{SOC} constraints on their moments to the original framework \cite{tacchi2022convergence}.

Let $M, m_0 \in \N$ be positive integers. For $i = 1..M$, let $m_i \in \N$ and $X_i \subset \R^{m_i}$ be a compact set. Let:
\begin{itemize}
    \item $\cX_0 \subset \Sym^{m_0}$ be a vector space of symmetric matrices. Specific instances of $\cX_0$ could be $\{\mathrm{diag}(\chi) \; | \; \chi \in \R^{m_0}\}$ (the space of diagonal matrices, corresponding to linear programming), or $\cX_0 = \Sym^{m_0}$ (the space of all symmetric matrices, corresponding to semidefinite programming). In particular, there exists such a space to represent second order cone programming \cite{alizadeh2003second}),
    \item $\cX_\infty = \mathcal{M}(X_1)\times \ldots \times \mathcal{M}(X_M)$ be a vector space of signed Borel measures, equipped with its weak-$*$ topology, so that its topological dual is $\cX^*_\infty = C(X_1)\times\ldots\times C(X_M)$,
    \item $\cX = \cX_0 \times \cX_\infty$ be our decision space, with topological dual $\cX^* = \cX_0 \times \cX_\infty^*$,
    \item A Banach space $\cY$ with dual $\cY^*$ that will represent our constraint space for equality constraints. In the context of the moment-SOS hierarchy, $\cY$ is chosen as a product space of smooth/polynomial functions defined on compact sets,
    \item $\cX_+ = \{(X,\mu_1,\ldots,\mu_M) \in \cX \; | \; X \succeq 0, \quad \forall i = 1..M, \mu_i \in \Mp{X_i}\}$ and
    \item[] $\cX^*_+  =\{(Y,v_1,\ldots,v_M) \in \cX^* \; | \; Y \succeq 0, \quad \forall i=1..M, v_i \geq 0 \}$ be two convex cones.
\end{itemize}
For $\phi = (X,\mu_1,\ldots,\mu_M) \in \cX$ and $\psi = (Y,v_1,\ldots,v_M) \in \cX^*$, we define the duality
\begin{equation}
    \langle \psi, \phi \rangle_\cX = \Tr{X \, Y} + \sum_{i=1}^M \int_{X_i} v_i(x_i) d\mu_i(x_i).
\end{equation}
Similarly, we denote by $\langle \cdot , \cdot \rangle_\cY$ the duality between elements of $\cY$ and $\cY^*$. Let $A : \cX \longrightarrow \cY^*$ be a continuous linear map, $y \in \cY^*$ be a vector of continuous linear forms (when $\cY$ is made of polynomials, $y$ is a moment sequence), $C \in \cX_0$, $p \in \R[x_1]\times\ldots\times\R[x_M] \subset \cX^*_\infty$ be a vector of polynomials $\gamma = (C,g) \in \cX^*$. We consider the following moment-SDP problem:
\begin{subequations} \label{eq:msdp}
\begin{align}
    p^*_M = & \sup \quad \langle \gamma , \phi \rangle_\cX, &  \phi &\in \cX_+, & A \phi &= y  \label{eq:pmsdp}
\end{align}
with dual problem
\begin{align}
    d^*_M = & \inf \quad \langle w , y \rangle_\cY, &  w &\in \cY,    & A^\dagger w - \gamma &\in \cX^*_+. \label{eq:dmsdp} 
\end{align}
\end{subequations}
It is well known that weak duality $p^*_M \leq d^*_M$ always holds \cite{barvinok2002convex}. In this section, we will prove that under some mild assumptions, strong duality $p^*_M = d^*_M$ also holds.

We first prove a simple lemma on strong duality conditions.

\begin{lem} \label{lem:strongdual}
In this lemma, we consider again the duality pair \eqref{eq:msdp}, but with generic spaces $\cX$, $\cY$ and a convex cone $\cX_+$. We also define, for any vector space $\mathcal{Z}$ containing some vector $\eta$ and any linear map $U : \cX \longrightarrow \mathcal{Z}$, the level set $U_\eta = \{\phi \in \cX \; | \; U \phi = \eta\}$. In such setting, we assume that
\begin{enumerate}
    \item[A1'] $\exists \phi \in \cX_+$ such that $A \phi = y$.
    \item[A2'] $A_0 \cap \gamma_0 \cap \cX_+ = \{0\}$.
    \item[A3'] $\exists \psi \in \cX^*$ such that
    \begin{enumerate}
        \item $\langle \psi , \cX_+ \rangle_\cX \subset \R_+$
        \item $\psi_0 \cap \cX_+ = \{0\}$ 
        \item $\psi_1 \cap \cX_+$ is compact.
    \end{enumerate}
\end{enumerate}
\begin{center}
    Then, $p^*_M = d^*_M.$
\end{center}
Moreover, if $p^*_M < \infty$, then there is an optimal $\phi^*$ feasible for \eqref{eq:pmsdp}  such that $\langle \gamma , \phi^* \rangle_\cX = p^*_M$.
\end{lem}
\begin{proof}
    We use \cite[Chap. IV: Thm (7.2), Lem (7.3)]{barvinok2002convex}. Consider the cone
   $$\cX_A^\gamma = \{(A\phi , \langle\gamma,\phi\rangle_\cX) \; | \; \phi \in \cX_+\}.$$
   Theorem (7.2) of \cite{barvinok2002convex} ensures that under A1' and closedness of $\cX_A^\gamma$, strong duality holds, and that $p^*_M < \infty$ then implies existence of an optimal $\phi^*$. Then, Lemma (7.3) of \cite{barvinok2002convex} states that if $\cX_+$ has a compact convex base and A2' holds, then $A_\cX^\gamma$ is closed. Thus, we need to find a compact convex base of $\cX_+$.

Let $\phi \in \cX_+ \setminus \{0\}$. A3'.\textit{(a)-(b)} ensure that $\langle \psi , \phi \rangle_\cX > 0$ so that $\tilde\phi = \frac{\phi}{\langle\psi , \phi \rangle_\cX}$ is well defined and belongs to the cone $\cX_+ \setminus \{0\}$. Moreover, $\langle \psi ,  \tilde\phi\rangle_\cX = 1$ is clear by definition, so that $\tilde\phi \in \psi_1 \cap \cX_+$. This proves that any $\phi \in \cX_+ \setminus \{0\}$ can be described as $\phi = \langle \psi , \phi \rangle_\cX \tilde\phi$ with $\tilde\phi \in \psi_1 \cap \cX_+$ and $\langle \psi , \phi \rangle_\cX > 0$, which is the definition of $\psi_1 \cap \cX_+$ being a base of $\cX$. By compactness assumption A3'.\textit{(c)}, we deduce that the assumptions of Lemma (7.3) of Theorem (7.2) of \cite{barvinok2002convex}  hold: $X_A^\gamma$ is closed and thus $p^*_M = d^*_M$.
\end{proof}

\begin{thm}
\label{thm:strongdual}
Suppose that there exists $B>0$ such that for all $\phi = (X,\mu_1,\ldots,\mu_M)$ feasible for \eqref{eq:pmsdp}, one has $\Tr{X^2} \leq B^2$ and for $i=1..M$, $\langle 1 , \mu_i\rangle \leq B$. Also assume that at least one such feasible $\phi$ exists. Then, $p^*_M = d^*_M$. Moreover, there exists an optimal $\phi^*$ such that $A\phi^* = y$ and $\langle \gamma , \phi^* \rangle_\cX = p^*_M$.
\end{thm}
\begin{proof}
We prove that the assumptions of Lemma \ref{lem:strongdual} hold. First of all, $\cX_+$ is indeed a convex cone under A1', as it is the product of convex cones $\Sym_+^{m_0}$ and $\Mp{X_i}$.

Next, we focus on hypothesis A2' Let $\phi = (X,\mu_1,\ldots,\mu_M) \in A_0 \cap \gamma_0 \cap \cX_+$. We want to prove that $\phi = 0$. Let $\phi^{(0)} = (X^{(0)}, \mu_1^{(0)},\ldots,\mu_M^{(0)}) \in \cX_+$ such that $A \phi^{(0)} = y$. Define, for $t \geq 0$, $\phi^{(t)} = \phi^{(0)} + t \phi$. Let $t \geq 0$. Since $\cX_+$ is a convex cone, $\phi^{(t)} \in \cX_+$. In addition,

$$A \phi^{(t)} = A \phi^{(0)} + t \, A \phi = A \phi^{(0)} = y,$$

so that $\phi^{(t)}$ is feasible for \eqref{eq:pmsdp}. Thus, by assumption,

\begin{align*}
    B^2 & \geq \Tr{(X^{(0)} + t X)^2} \\
    & = \Tr{X^{(0)2} + 2t X^{(0)} X + t^2 X^2} \\
    & = \Tr{X^{(0)2}} + 2t \Tr{X^{(0)} X} + t^2 \Tr{X^2} \\
    & = t^2 \Tr{X^2} + \underset{t\to\infty}{o}(t^2).
\end{align*}
Staying bounded when $t$ goes to infinity requires that $\Tr{X^2} = 0$, implying that $X=0$. The same reasoning replacing $\Tr{X^2}$ with $\langle 1 , \mu_i \rangle$ yields that for all $i=1..M$, $\mu_i=0$. Thus, $\phi=0$ and A2' holds.

We turn to A3' and consider $\psi = (I_{m_0},1_M)$ where $I_{m_0}$ is the size $m_0$ identity matrix and $1_M$ is the dimension $M$ vector of functions that are all constant equal to $1$. Note that if $(X,\mu_1,\ldots,\mu_M) \in \cX_+$, then $\Tr{X} \geq 0$ and  $\forall i=1..M: \ \langle 1 , \mu_i \rangle \geq 0$ (i.e. $\langle \psi , \cX_+ \rangle_\cX \subset \R_+$). Moreover the equality cases in those inequalities only hold for $X=0$ and $\mu_i=0$ respectively, so that $\psi_0 \cap \cX_+ = \{0\}.$ It only remains to prove that $\psi_1 \cap \cX_+$ is compact.

First, it is bounded for the norm $\norm{(X,\mu_1,\ldots,\mu_M)} = \sqrt{\Tr{X^2}} + \sum_{i=1}^M \norm{\mu_i}_{TV}$ where the total variation norm of a signed measure is $$\norm{\mu}_{TV} = \sup\{\langle v , \mu \rangle \; | \; -1 \leq v \leq 1\}.$$ In particular, the total variation norm of a nonnegative measure $\mu \in \Mp{X}$ is equal to its mass $\norm{\mu}_{TV} = \langle 1 , \mu \rangle$.

Indeed, let $(X,\mu_1,\ldots,\mu_M) \in \psi_1 \cap \cX_+$; then $\Tr{X}\leq 1$ and for $i=1..M$, $1 \geq \langle 1 , \mu_i \rangle = \norm{\mu_i}_{TV}$. As $X$ is positive semidefinite, $\Tr{X}\leq 1$ means that none of the eigenvalues of $X$ are bigger than $1$. As such, $\Tr{X^2} \leq m_0$, because it is the sum of the squares of the eigenvalues of $X$. Thus, for all $\phi \in \psi_1 \cap \cX_+$, $\norm{\phi} \leq m_0 + M$.

Then, it is also closed for the weak-* topology of $\cX$ by continuity of $\langle \psi , \cdot \rangle_\cX$ and closedness of $\cX_+$ as the product of closed sets $\Sym^{m_0}_+$ and $\Mp{X_i}$. Thus, $\psi_1 \cap \cX_+$ is weak-* closed and bounded: according to the Banach-Alaoglu theorem, it is compact: assumption A3' of Lemma \ref{lem:strongdual} holds, which concludes the proof of strong duality.

Finally, let $\phi = (X,\mu_1,\ldots,\mu_M) \in \cX_+$ feasible for \eqref{eq:pmsdp}. Then, by assumption, $\Tr{X^2} \leq B^2$ and $\langle 1 , \mu_i \rangle \leq B$ for all $i=1..M$. Thus, one has
\begin{align*}
    \langle \gamma , \phi \rangle_\cX & = \Tr{C\, X} + \sum_{i=1}^M \langle g_i , \mu_i \rangle. \\
    \intertext{Using the Cauchy-Schwarz inequality results in }
    & \leq \sqrt{\Tr{C^2} \Tr{X^2}} + \sum_{i=1}^M \langle g_i , \mu_i \rangle \\    
    & \leq \sqrt{\Tr{C^2} \Tr{X^2}} + \sum_{i=1}^M \sup_{X_i}|g_i| \langle 1 , \mu_i \rangle \\
    & \leq \left(\sqrt{\Tr{C^2}} + \sum_{i=1}^M \norm{g_i}_\infty \right) B
\end{align*}
so that taking the supremum over all feasible $\phi$ yields
$$ p^*_M \leq \left(\sqrt{\Tr{C^2}} + \sum_{i=1}^M \norm{g_i}_\infty \right) B < \infty, $$
which is the last hypothesis of Lemma \ref{lem:strongdual} and ensures existence of an optimal solution. 
\end{proof}
\section{Strong Duality of
Chance-Peak Concentration-Bound Linear Programs}
\label{app:duality_chance}

In order to apply the strong duality results of Appendix \ref{app:duality_general} towards the tail-bound chance-peak problem (Theorem \ref{thm:cdc_strong_dual}), we need to provide an \ac{SDP}-representation of the \ac{SOC} cone.  
\begin{lem} \label{lem:soc_lmi}
An element $(y,s) \in \mathbb{L}^n$ satisfies the \ac{LMI} \cite{alizadeh2003second}
\begin{align}
    \Phi = \begin{bmatrix}
    s & y^T \\ y & s I_n
    \end{bmatrix} \succeq 0. \label{eq:soc_embedding}
\end{align}


\end{lem}
\begin{proof}
When ${s}=0$, the containment $({y}, 0) \in {\mathbb{L}}^n$ requires that ${y} = \0_n$. The matrix in $\Phi$ \eqref{eq:soc_embedding} is therefore the $\0_{n \times n}$ matrix which is \ac{PSD}. 
Now consider the case where ${s} > 0$. A Schur complement of $\Phi$ yields the constraint ${s - y^T y / s} \geq 0$. Multiplying through by the positive ${s}$ results in ${s^2 - \norm{y}^2_2} \geq 0, \ {s} > 0$, which is the definition of the \ac{SOC} cone ${(y, s) \in \mathbb{L}^n}$. Lemma \ref{lem:soc_lmi} is therefore proven.
\end{proof}

Lemma \ref{lem:soc_lmi} ensures that problem \eqref{eq:var_meas_soc_ext} is an instance of the more generic \eqref{eq:pmsdp} with $M = 2$, $X_1 = X_2 = [0,T]\times X$, $\cY = C^2([0,T]\times X) \times \R^3$ and \begin{equation}
    \cX_0 = \left\{ { \begin{bmatrix}
    s & y^T \\ y & s I_3
    \end{bmatrix}} \; \middle| \; { s \in \R, y \in \R^3 } \right\}.\end{equation}
Therefore, we only need to verify that the hypotheses of Theorem \ref{thm:strongdual} hold in our specific case. 

Letting ${z}=([{z}_1, {z}_2, {z}_3], {z}_4) \in \mathbb{L}^3$ be an \ac{SOC}-constrained variable, we define the matrix $\Phi$ from \eqref{eq:soc_embedding} as
\begin{align}
    \Phi = \begin{bmatrix}
    {z}_4 & {z}_1 & {z}_2 & {z}_3 \\
    {z}_1 & {z}_4 & 0 & 0 \\
    {z}_2 & 0 & {z}_4 & 0 \\
    {z}_3 & 0 & 0 & {z}_4
    \end{bmatrix}.\label{eq:soc_embedding_q}
\end{align}

Theorem \ref{thm:strongdual} requires the following sufficient conditions to prove strong duality between \eqref{eq:var_meas_soc_ext} and \eqref{eq:var_cont_soc} and their optimality obtainment. 
\begin{enumerate}
\item[R1]There exists a feasible solution for $(\mu_\tau, \mu, {z})$ from \eqref{eq:var_meas_soc_ext}.
    \item[R2] The measures $\mu_\tau, \mu$ are bounded.
    \item[R3] The square of the matrix $\Phi$ from \eqref{eq:soc_embedding_q} has bounded trace.
\end{enumerate}

We start with R1. Letting $x(t \mid x_0)$ be an \ac{SDE} trajectory from \eqref{eq:sde_sol} and $t^* \in (0, T]$ be a stopping time with $\tau^* = t^* \wedge \tau_X$, we define $\mu$ as the occupation measure of $x(t \mid x_0)$ and $\mu_\tau$ as its time-$\tau^*$ state distribution $\mu_{\tau^*}$. Feasible choices for entries of the \ac{SOC}-constrained ${z}$  are (from Lemma \ref{lem:sqrt})
\begin{align}
    {z}_1 &= 1-\inp{p^2}{\mu_{t^*}}, & {z}_2 &= \sqrt{\inp{p^2}{\mu_{t^*}} - \inp{p}{\mu_{t^*}}^2}, \\
    {z}_3 &= 2\inp{p}{\mu_{t^*}}, & {z}_4 &= 1+\inp{p^2}{\mu_{t^*}}.
\end{align}

Requirement R2's satisfaction follows the statement in Lemma \ref{lem:moment_bound}  that $\mu_\tau, \mu$ are bounded under A1-A3.

We end with R3. The trace $\Tr{\Phi^2}=\sum_{ij} \Phi_{ij}^2 $ is equal to
\begin{subequations}
\label{eq:trace_square}
\begin{align}
    \Tr{\Phi^2} &=  2{z}_1^2 + 2{z}_2^2 + 2{z}_3^2 +4 {z}_4^2 \\&
    = 2(1-\inp{p^2}{\mu_{t^*}})^2 + 2( \sqrt{\inp{p^2}{\mu_{t^*}} - \inp{p}{\mu_{t^*}}^2})^2  \nonumber\\
    &\qquad + 2(2\inp{p}{\mu_{t^*}})^2 +4(1+\inp{p^2}{\mu_{t^*}})^2 \\
    &= 6 + 6\inp{p}{\mu_{t^*}}^2 + 6\inp{p^2}{\mu_{t^*}} + 6 \inp{p^2}{\mu_{t^*}}^2.
\end{align}
\end{subequations}

Let $\Pi_1 = \max_{x \in X} p(x)$ and  $\Pi_2 = \max_{x \in X} p(x)^2$ be bounds on $p$ and $p^2$ in $X$. Both $\Pi_1$ and $\Pi_2$ will be finite by the compactness of $X$ (A1) and the continuity of $p$ within $X$ (A3). Given that $\mu_{t^*}$ is a probability distribution supported in $X$, the moments of $\mu_{t^*}$ will be bounded by $\inp{p}{\mu_{t^*}} \leq \Pi_1$ and $\inp{p^2}{\mu_{t^*}} \leq \Pi_2$. The squared-trace in \eqref{eq:trace_square} can be upper-bounded by a finite value $B^2$ such that
\begin{align}
    \Tr{\Phi^2} \leq 6(1 +  \Pi_1^2 + \Pi_2 + \Pi_2^2) = B^2 < \infty. \label{eq:trace_square_B}
\end{align}

The finite bound $B^2 \in [0, \infty)$ from \eqref{eq:trace_square_B} validates R3, and completes all conditions necessary for Theorem \ref{thm:strongdual} to provide for strong duality and optima attainment.


\section{Strong Duality of
Chance-Peak Expected Shortfall Linear Programs}
\label{app:duality_cvar}

We follow the steps and conventions of Theorem 2.6 of \cite{tacchi2021thesis} to perform this proof of strong duality.

We collect the groups of variables into 
\begin{align}
    \bbmu &= (\mu_\tau, \mu, {\nu, \hat{\nu}}) &   \bell &= ({u,v,w}). \label{eq:bbmu}
\end{align}

We now define the following variable spaces
\begin{align}
    \mathcal{X}' &=  C([0, T]\times X)^2 \times C(\R)^2 \label{eq:dual_spaces}\\
    \mathcal{X} &= \mathcal{M}([0, T]\times X)^2 \times \mathcal{M}(\R)^2,\nonumber
\end{align}
and note that their nonnegative subcones are
\begin{align}
    \mathcal{X}_+' &= C_+([0, T]\times X)^2 \times C_+(\R) ^2\label{eq:dual_cones}\\
    \mathcal{X}_+ &= \Mp{[0, T]\times X}^2 \times \Mp{\R}^2. \nonumber
\end{align}
Additionally, the measure $\bbmu$ from \eqref{eq:bbmu} obeys $\bbmu \in \mathcal{X}_+$.
Assumption A1 imposes that the cones $\mathcal{X}_+'$ and $\mathcal{X}_+$ in \eqref{eq:dual_cones} form a pair of topological duals.

We define the constraint spaces $\mathcal{Y}$ and $\mathcal{Y}'$ as
\begin{align}
    \mathcal{Y}' &= {\R \times \cs([0, T] \times X) \times C(\R)} \\
    \mathcal{Y} &= {0 \times \cs([0, T] \times X)' \times \mathcal{M}(X)}.
\end{align}

We follow the convention of \cite{tacchi2021thesis} and define $\mathcal{Y}_+ = {\{0_\mathcal{Y}\}}$ and $\mathcal{Y}'_+ = \mathcal{Y}'$. The variable $\bell$ from \ref{eq:bbmu} satisfies $\bell \in \mathcal{Y}'$.

We formulate the affine maps of
\begin{align}
    \mathcal{A}(\bbmu) &= [{\inp{1}{\nu}}, \ \mu_\tau - \Lie^\dagger \mu, \ \epsilon {\nu + \hat{\nu}} - p_\# \mu_\tau] \\
    \mathcal{A}^*(\bell) &= [{ v - w \circ p , \ -\Lie v, \ u + \epsilon w, \ w} ],\nonumber
\end{align}
and the associated cost/constraint data 
\begin{subequations}
    \begin{align}
        \mathbf{b} &= [1, \ \delta_0 \otimes \mu_0, \ 0] & 
        \mathbf{c} &= [0, \ 0, \ {\mathrm{id}_\R}, \ 0].
    \end{align}
\end{subequations}

We point out that $\A$ and $\A^*$ are linear adjoints, that $\mathcal{X}$ has a weak-* topology, and that $\mathcal{Y}$ has a sup-norm-bounded weak topology. We also note that the following pairings  satisfy 
\begin{subequations}
\begin{align}
\inp{\mathbf{c}}{\boldsymbol{\mu}} &= \inp{\mathrm{id}_\R}{\nu} \\
    \inp{\boldsymbol{\ell}}{\mathbf{b}} &= {u} + \inp{v(0, {\bullet})}{\mu_0}.
\end{align}
\end{subequations}

Problem \eqref{eq:peak_cvar_meas} may be expressed  as 
\begin{align}
    p^* =& \sup_{\boldsymbol{\mu} \in \mathcal{X}_+} \inp{\mathbf{c}}{\boldsymbol{\mu}} & & \mathbf{b} - \A(\boldsymbol{\mu}) \in \mathcal{Y}_+. \label{eq:dist_meas_std}\\
\intertext{Problem \eqref{eq:peak_cvar_cont} may be converted into }
    d^* = &\inf_{\boldsymbol{\ell} \in \mathcal{Y}'_+} \inp{\boldsymbol{\ell}}{\mathbf{b}}
    & &\A'(\boldsymbol{\ell}) - \mathbf{c} \in \mathcal{X}_+. \label{eq:dist_cont_std}
\end{align}

Sufficient conditions for strong duality from \cite[Theorem 2.6]{tacchi2021thesis} are:
\begin{itemize}
    \item[R1] Masses of measures in $\bbmu$ satisfying $b-\mathcal{A}(\bbmu) \in \mathcal{Y}_+$ are bounded.
    \item[R2] There exists a feasible $\bbmu^f$ with $b-\mathcal{A}(\bbmu^f) \in \mathcal{Y}_+$.
    \item[R3] Functions involved in defining $\mathbf{c}$, $\mathbf{b}$,  and $\mathcal{A}$ are continuous.
\end{itemize}

Requirement R1 is satisfied by Theorem \ref{thm:bounded_mass_cvar} under A1, A3, and A4. Requirement R2 holds because there exists a feasible solution $\bbmu^f$ from the proof of Theorem \ref{thm:upper_bound_cvar} (constructed from a feasible trajectory of the stochastic process) under A4. Requirement R3 is fulfilled because $v \in \cs([0, T] \times X)$ implies that $\Lie v \in C([0, T] \times X)$ (A2) and $p$ is continuous (A3). All requirements are fulfilled, which proves strong duality between \eqref{eq:peak_cvar_meas} and \eqref{eq:peak_cvar_cont}.



\end{document}